\documentclass[a4paper,10pt]{amsart}
\usepackage{amsmath,amssymb}
\usepackage{amscd}
\newtheorem{thm}{Theorem}[section]
\newtheorem{prop}[thm]{Proposition}
\newtheorem{lemma}[thm]{Lemma}
\newtheorem{cor}[thm]{Corollary}
\newtheorem{exam}[thm]{Example}
\theoremstyle{remark}
\newtheorem{remark}[thm]{Remark}
\newcommand{\id}{{\rm{id}}}
\newcommand{\Ad}{{\rm{Ad}}}

\newcommand{\BN}{\mathbf N}

\newcommand{\BC}{\mathbf C}

\newcommand{\BK}{\mathbf K}
\newcommand{\BB}{\mathbf B}

\newcommand{\la}{\langle}
\newcommand{\ra}{\rangle}
\newcommand{\Ind}{{\rm{Ind}}}

\newcommand{\Pic}{{\rm{Pic}}}
\newcommand{\Equi}{{\rm{Equi}}}
\newcommand{\Aut}{{\rm{Aut}}}
\newcommand{\Int}{{\rm{Int}}}

\newtheorem{Def}{Definition}[section]

\title{Strong Morita equivalence for conditional expectations}
\author{Kazunori Kodaka}
\address{Department of Mathematical Sciences, Faculty of Science, Ryukyu University, 
\endgraf
Nishihara-cho, Okinawa, 903-0213, Japan}
\address{\sl{E-mail address}: \rm{kodaka@math.u-ryukyu.ac.jp}}



\begin{document}
\begin{abstract}
We consider two inclusions of $C^*$-algebras whose small $C^*$-algebras have approximate units of the
large $C^*$-algebras and their two spaces of all bounded bimodule linear maps.
We suppose that the two inclusions of $C^*$-algebras are strongly Morita equivalent.
In this paper, we shall show that there exists an isometric isomorphism from one of the spaces of all bounded
bimodule linear maps to the other space and we shall study on basic properties about the isometric isomorphism. And,
using this isometric isomorphism, we define the Picard group for a bimodule linear map and discuss on the
Picard group for a bimodule linear map.
\end{abstract}

\maketitle

\section{Introduction}\label{sec:intro} Let $A\subset C$ and $B\subset D$ be inclusions of
$C^*$-algebras with $\overline{AC}=C$,
$\overline{BD}=D$. Let ${}_A \BB_A (C, A)$, $_{B} \BB_B (D, B)$ be the spaces of all bounded
$A-A$-bimodule linear maps and all bounded $B-B$-bimodule linear maps from $C$ and $D$ to $A$ and $B$, respectively.
We suppose that they are strongly Morita equivalent with respect to a
$C-D$-equivalence bimodule $Y$ and its closed subspace $X$.
In this paper, we shall
define an isometric isomorphism $f$ of ${}_B \BB_B (D, B)$ onto ${}_A \BB_A (C, A)$ induced by
$Y$ and $X$. We shall study on the basic
properties about $f$. Especially, we shall give the following result: Let $A$ and $B$ are
unital $C^*$-algebras. If $\phi$ is an element in ${}_B \BB_B (D, B)$
having a quasi-basis defined in Watatani \cite [Definition 1.11.1]{Watatani:index}, then $f(\phi)$ is also an
element in ${}_A \BB_A (C, A)$ having a quasi-basis and there is an isomorphism of $\pi$ of $B' \cap D$ onto
$A' \cap C$ such that
$$
\theta^{f(\phi)}=\pi\circ\theta^{\phi}\circ\pi^{-1} ,
$$
where $\theta^{\phi}$ and $\theta^{f(\phi)}$ are the modular automorphisms for $\phi$ and $f(\phi)$,
respectively which are defined in \cite [Definition 1.11.2]{Watatani:index}.
We note that the isometric isomorphism $f$ of ${}_B \BB_B (D, B)$ onto ${}_A \BB_A (C, A)$ depends on the
choice of a $C-D$-equivalence bimodule $Y$ and its closed subspace $X$.
We shall also discuss the relation between $f$ and the pair $(X, Y)$.
Furthermore, we defne the Picard group for a bimodule linear map and
we discuss on the Picard group of a bimodule linear map.
\par
For a $C^*$-algebra $A$, we denote by $\id_A$ the identity
map on $A$ and if $A$ is unital, we denote by$1_A$ the unit element in $A$.
If no confusion arises, we denote them by $1$ and $\id$, respectively.
For each $n\in\BN$, we denote by $M_n (\BC)$ the $n\times n$-matrix algebra over $\BC$ and $I_n$
denotes the unit element in $M_n (\BC)$. Also, we denote by $M_n (A)$ the $n\times n$-matrix algebra
over $A$ and we identify $M_n(A)$ with $A\otimes M_n (\BC)$ for any $n\in \BN$.
For a $C^*$-algebra $A$, let $M(A)$ be the multiplier $C^*$-algebra of $A$.
\par
Let $\BK$ be the $C^*$-algebra of all compact operators on a countably infinite dimensional Hilbert space.
\par
Let $A$ and $B$ be $C^*$-algebras. Let $X$ be an $A-B$-equivalence bimodule.
For any $a\in A$, $b\in B$, $x\in X$,
we denote by $a\cdot x$ the left $A$-action on $X$ and by $x\cdot b$ the right $B$-action on $X$, respectively.
Let ${}_A \BK(X)$ be the $C^*$-algebra of all `` compact" adjointable left $A$-linear operators on $X$
and we identify ${}_A \BK(X)$ with $B$. Similarly we define $\BK_B (X)$ and we identify $\BK_B (X)$ with $A$.

\section{Construction}\label{sec:con} Let $A\subset C$ and $B\subset D$ be
inclusions of $C^*$-algebras with $\overline{AC}=C$ and $\overline{CD}=D$.
Let ${}_A \BB_A (C, A)$, $_{B} \BB_B (D, B)$ be the spaces of all bounded
$A-A$-bimodule linear maps and all bounded $B-B$-bimodule linear maps from $C$ and $D$ to $A$ and $B$, respectively.
We suppose that $A\subset C$ and
$B\subset D$ are strongly Morita equivalent with respect to a $C-D$-equivalence bimodule $Y$ and
its closed subspace $X$. We construct an isometric isomorphism of ${}_B \BB_B (D, B)$ onto ${}_A \BB_A (C, A)$.
For any $\phi\in {}_B \BB_B (D, B)$, we define the linear
map $\tau$ from $Y$ to $X$ by
$$
\la x \, , \, \tau(y) \ra_B =\phi( \la x, y \ra_D )
$$
for any $x\in X$, $y\in Y$.

\begin{lemma}\label{lem:con1} With the above notation, $\tau$ satisfies the following conditions:
\newline
$(1)$ $\tau(x\cdot d)=x\cdot \phi(d)$,
\newline
$(2)$ $\tau(y\cdot b)= \tau(y)\cdot b$,
\newline
$(3)$ $\la x \, , \, \tau(y) \ra_B =\phi( \la x, y \ra_D )$
\newline
for any $b\in B$, $d\in D$, $x\in X$, $y\in Y$.
Also, $\tau$ is bounded and $||\tau||\leq ||\phi||$.
Furthermore, $\tau$ is the unique linear map from $Y$ to $X$ satisfying Condition $(3)$.
\end{lemma}
\begin{proof} Except for the uniquness of $\tau$, we can prove this lemma in the same way as in the proof of
\cite [Lemma 3.4 and Remark 3.3 (ii)]{KT4:morita}. Indeed, the definition of $\tau$, $\tau$ satisfies Condition (3). For any
$x, z\in X$, $d\in D$,
$$
\la z, \tau(x\cdot d)\ra_B =\phi(\la z, x\cdot d \ra_D )=\phi(\la z, x \ra_B \, d) =\la z, x \ra_B \, \phi(d)
=\la z, x\cdot \phi(d)\ra_B .
$$
Hence $\tau(x\cdot d)=x\cdot \phi(d)$ for any $x\in X$, $d\in D$. Also, for any $b\in B$, $z\in X$ $y\in Y$,
$$
\la z, \tau(y\cdot b)\ra_B =\phi(\la z, y\cdot b \ra_D \, )=\phi(\la z, y \ra_D \, b)=\phi(\la z, y \ra_D \, )b
=\la z, \tau(y)\cdot b \ra_B .
$$
Thus $\tau(y\cdot b)=\tau(y)\cdot b$ for any $b\in B$, $y\in Y$.
By Raeburn and Williams \cite [the proof of Lemma 2.18]
{RW:continuous}, for any $y\in Y$
\begin{align*}
||\tau(y)|| & =\sup \{||\la x, \tau(y) \ra_B \,\,  || \, \, | \, ||x||\leq 1, \, x\in X \} \\
& =\sup \{||\phi(\la x, y \ra_D \, ) \, || \, \, | \, \, ||x||\leq 1, \, x\in X \} \\
& \leq \sup \{||\phi|| \, ||x|| \, ||y||  \, \, | \, \, ||x||\leq 1, \, x\in X \} \\
& \leq ||\phi||\, ||y||.
\end{align*}
Thus $\tau$ is bounded and $||\tau||\leq ||\phi||$. Furthermore, let $\tau'$ be a linear map from $Y$ to $X$ satisfying
Condition (3). Then for any $x\in X$, $y\in Y$,
$$
\la x, \tau(y)\ra_B =\phi(\la x, y \ra_D \, )=\la x, \tau' (y) \ra_B .
$$
Hence $\tau(y)=\tau' (y)$ for any $y\in Y$. Therefore, $\tau$ is unique.
\end{proof}

\begin{lemma}\label{lem:con2} With the above notation, $\tau(a\cdot y)=a\cdot\tau(y)$ for any $a\in A$, $y\in Y$.
\end{lemma}
\begin{proof} We can prove this lemma by routine computations.
Indeed, for any $x, z\in X$, $y\in Y$,
$$
\tau({}_A \la x, z \ra\cdot y)=\tau( x\cdot \la z , y \ra_D )=x\cdot \phi(\la z, y \ra_D )
=x\cdot \la z, \tau (y)\ra_B ={}_A \la x, z \ra \cdot \tau(y) .
$$
Since $\overline{{}_A \la X\, , \, X \ra}=A$ and $\tau$ is bounded, we obtain the conclusion.
\end{proof}

Let $\psi$ be the linear map from $C$ to $A$ defined by
$$
\psi(c)\cdot x=\tau(c\cdot x)
$$
for any $c\in C$, $x\in X$, where we identify $\BK_B (X)$ with $A$ as $C^*$-algebras
by the map $a\in A \mapsto T_a \in \BK_B (X)$, which is defined by $T_a (x)=a\cdot x$ for any $x\in X$.

\begin{lemma}\label{lem:con3} With the above notation, $\psi$ is a linear map from $C$ to $A$
satisfying the following conditions:
\newline
$(1)$ $\tau(c\cdot x)=\psi(c)\cdot x$,
\newline
$(2)$ $\psi({}_C \la y, x \ra)={}_A \la \tau(y) \, , \, x \ra$
\newline
for any $c\in C$, $x\in X$, $y\in Y$.
Also, $\psi$ is a bounded $A-A$-bimodule linear map from $C$ to $A$ with
$||\psi||\leq ||\tau||$. Furthermore, $\psi$ is the unique linear map from $C$ to $D$ satisfying Condition $(1)$.
\end{lemma}
\begin{proof} We can prove this lemma in the same way as in the proof of
\cite [Proposition 3.6]{KT4:morita}. Indeed, by the definition of $\psi$, $\psi$ satisfies Condition (1). Also,
for any $x, z\in X$, $y\in Y$,
$$
\psi({}_C \la y, x \ra )\cdot z =\tau(\, {}_C \la y, x \ra \cdot z ) =\tau(y\cdot \la x, z \ra_B \, ) =\tau(y)\cdot \la x, z \ra_B
={}_A \la \tau(y) \, , \, x \ra \cdot z
$$
Hence $\psi(\, {}_C \la y , x \ra )={}_A \la \tau(y) \, , \, x \ra$ for any $x\in X$, $y\in Y$. For any
$c\in C$,
\begin{align*}
||\psi(c)|| & =\sup \{||\psi(c)\cdot x || \,\,  | \,\,  ||x||\leq 1, \, x\in X \} \\
& =\sup \{||\tau(c\cdot x) || \,\,  | \,\,  ||x||\leq 1, \, x\in X \} \\
& \leq \sup \{||\tau|| \, ||c|| \, ||x|| \,\,  | \,\,  ||x||\leq 1, \, x\in X \} \\
& =||\tau||\, ||c|| .
\end{align*}
Thus $\psi$ is bounded and $||\psi||\leq||\tau||$. Next, we show that $\psi$ is an $A-A$-bimodule map
from $C$ to $A$. It suffices to show that
$$
\psi(ac)=a\psi(c), \quad \psi(ca)=\psi(c)a
$$
for any $a\in A$, $c\in C$. Fo any $a\in A$, $c\in C$, $x\in X$,
$$
\psi(ac)\cdot x=\tau(ac\cdot x)=\tau(a\cdot (c\cdot x))=a\tau(c\cdot x)=a\psi(c)\cdot x
$$
by Lemma \ref {lem:con2}. Hence $\psi(ac)=a\psi(c)$. Also,
$$
\psi(ca)\cdot x=\tau(ca\cdot x)=\tau(c\cdot (a\cdot x))=\psi(c)\cdot (a\cdot x)=\psi(c)a\cdot x
$$
since $a\cdot x\in X$. Hence $\psi(ca)=\psi(c)a$. Let $\psi'$ be a linear map from $C$ to $A$
satisfying Condition (1). Then for any $x\in X$, $c\in C$,
$$
\psi(c)\cdot x=\tau(c\cdot x)=\psi' (c)\cdot x .
$$
Thus $\psi(c)=\psi' (c)$ for any $c\in C$. Therefore, we obtain the conclusion.
\end{proof}

\begin{prop}\label{prop:con4} Let $A\subset C$ and $B\subset D$ be inclusions of
$C^*$-algebras with $\overline{AC}=C$ and $\overline{BD}=D$.
We suppose that $A\subset C$ and $B\subset D$ are strongly Morita equivalent with
respect to a $C-D$-equivalence bimodule $Y$ and its closed subspace $X$. Let $\phi$ be any element
in ${}_B \BB_B (D, B)$. Then there are the unique linear map $\tau$ from $Y$ to $X$ and the
unique element $\psi$ in ${}_A \BB_A (C, A)$ satisfying the following conditions:
\newline
$(1)$ $\tau(c\cdot x)=\psi(c)\cdot x$,
\newline
$(2)$ $\tau(a\cdot y)=a\cdot \tau(y)$,
\newline
$(3)$ ${}_A \la \tau(y) \, , \, x \ra =\psi( \, {}_C \la y, x \ra )$,
\newline
$(4)$ $\tau(x\cdot d)=x\cdot \phi(d)$,
\newline
$(5)$ $\tau(y\cdot b)=\tau(y)\cdot b$,
\newline
$(6)$ $\phi(\la x, y \ra_D )=\la x \, , \, \tau(y) \ra_B$
\newline
for any $a\in A$, $b\in B$, $c\in C$, $d\in D$, $x\in X$, $y\in Y$.
Furthermore, $||\psi||\leq ||\tau|| \leq ||\phi||$. Also, for any element $\psi\in {}_A \BB_A (C, A)$, we
have the same results as above. 
\end{prop}
\begin{proof} This is immediate by
Lemmas \ref{lem:con1}, \ref{lem:con2} and \ref{lem:con3}.
\end{proof}

We denote by $f_{(X, Y)}$ the map
$$
\phi\in{}_B \BB_B (D, B)\mapsto\psi\in{}_A \BB_A (C, A)
$$
as above.
By the definition of $f_{(X, Y)}$ and Proposition \ref{prop:con4}, we can see that $f_{(X, Y)}$ is an isometric
isomorphism of ${}_B \BB_B (D, B)$ onto ${}_A \BB_A (C, A)$.

\begin{lemma}\label{lem:con5} With the above notation, let $\phi$ be any element in ${}_B \BB_B (D, B)$.
Then $f_{(X, Y)}(\phi)$ is the unique linear map from $C$ to $A$ satisfying that
$$
\la x \, , \, f_{(X, Y)}(\phi)(c)\cdot z \ra_B =\phi(\la x \, , \, c\cdot z \ra_D )
$$
for any $c\in C$, $x, z\in X$.
\end{lemma}
\begin{proof} It is clear that $f_{(X, Y)}(\phi)$ satisfies the above equation by the definition of $f_{(X, Y)}(\phi)$ and
Lemma \ref{lem:con1}. Let $\psi$ be another linear map from $C$ to $A$ satisfying the above equation.
Then for any $c\in C$, $x, z\in X$, 
$$
\la x, \, f_{(X, Y)}(\phi)(c)\cdot z \ra_B=\la x, \, \psi(c)\cdot z \ra_B .
$$
Hence $f_{(X, Y)}(\phi)(c)=\psi(c)$ for any $c\in C$. Thus $f_{(X, Y)}(\phi)=\psi$.
\end{proof}

Let $\Equi(A, C, B, D)$ be the set of all pairs $(X, Y)$ such that $Y$ is a $C-D$-equivalence bimodule and $X$ is
its closed subspace satisfying Conditions (1), (2) in \cite [Definition 2.1]{KT4:morita}.
We define an equivalence relation ``$\sim$" in $\Equi (A, C, B, D)$ as follows:
For any $(X, Y), (Z, W)\in\Equi(A, C, B, D)$, we say that $(X, Y)\sim(Z, W)$ in $\Equi(A, C, B, D)$ if there is a $C-D$-
equivalence bimodule map $\Phi$ of $Y$ onto $W$ such that $\Phi|_X$ is a bijection of $X$ onto $Z$. Then $\Phi|_X$
is an $A-B$-equivalence bimodule isomorphism of $X$ onto $Z$
by \cite [Lemma 3.2]{Kodaka:Picard2}. We denote by $[X, Y]$
the equivalence class of $(X, Y)\in\Equi(A, C, B, D)$ and we denote by $\Equi(A, C, B, D)/\!\sim$ the set of all
equivalence classes $[X, Y]$ of $(X, Y)\in\Equi(A, C, B, D)$.

\begin{lemma}\label{lem:con6} With the above notation, let $(X, Y), (Z, W)\in\Equi(A, C, B, D)$ with
$(X, Y)\sim(Z, W)$ in $\Equi(A, C, B, D)$. Then $f_{(X, Y)}=f_{(Z, W)}$.
\end{lemma}
\begin{proof} Let $\Phi$ be a $C-D$-equivalence bomodule isomorphism of $Y$ onto $W$
satisfying that $\Phi|_X$ is a bijection of $X$ onto $Z$. Let $\phi\in{}_B\BB_B (D, B)$.
Then for any $x_1 , x_2 \in X$, $c\in C$,
\begin{align*}
\la \Phi(x_1 ) \, , \, f_{(Z, W)}(\phi)(c)\cdot \Phi(x_2 ) \ra_B & =
\phi(\la \Phi(x_1 ) \, , \, c\cdot \Phi(x_2 ) \ra_D \, ) \\
& =\phi(\la \Phi(x_1 ) \, , \, \Phi(c\cdot x_2 ) \ra_D \, ) \\
& =\phi(\la x_1 \, , \, c\cdot x_2 \ra_D \, ) \\
& =\la x_1 \, , \, f_{(X, Y)}(\phi)(c)\cdot x_2 \ra_B .
\end{align*}
On the other hand,
\begin{align*}
\la \Phi(x_1 )\, , \, f_{(Z, W)}(\phi)(c)\cdot \Phi(x_2 ) \ra_B \, & =
\la \Phi(x_1 ) \, , \, \Phi(f_{(Z, W)}(\phi)(c)\cdot x_2 ) \ra_B \\
& =\la x_1 \, , \, f_{(Z, W)}(\phi)(c)\cdot x_2 \ra_B .
\end{align*}
Hence we obtain that
$$
\la x_1 \, , \, f_{(X, Y)}(\phi)(c)\cdot x_2 \ra_B =\la x_1 \, , \, f_{(Z, W)}(\phi)(c)\cdot x_2 \ra_B .
$$
for any $c\in C$, $x_1 , x_2 \in X$.
Therefore, we obtain the conclusion.
\end{proof}

We denote by $f_{[X, Y]}$ the isometric isomorphism of ${}_B \BB_B (D, B)$
into ${}_A \BB_A (C, A)$ induced by the equivalence class $[X, Y]$ of $(X, Y)\in\Equi (A, C, B, D)$.
\par
Let $L\subset M$ be an inclusion of $C^*$-algebras with $\overline{LM}=M$, which is strongy
Morita equivalent to the inclusion $B\subset D$ with respect to a $D-M$-equivalence bomodule $W$ and
its closed subspace $Z$. Then the inclusion $A\subset C$ is strongly Morita equivalent to the
inclusion $L\subset M$ with respect to the $C-M$-equivalence bimodule $Y\otimes_D W$ and its closed
subspace $X\otimes_B Z$.

\begin{lemma}\label{lem:con7} With the above notation,
$$
f_{[X\otimes_B Z \, , \, Y\otimes_D W]}=f_{[X, Y]}\circ f_{[Z, W]} .
$$
\end{lemma}
\begin{proof} Let $x_1 , x_2 \in X$ and $z_1 , z_2 \in Z$. Let $c\in C$ and $\phi\in {}_L\BB_L (M, L)$. Then
\begin{align*}
& \la x_1 \otimes z_1 \, , \, (f_{[X, Y]}\circ f_{[Z, W]})(\phi)(c)\cdot x_2 \otimes z_2 \ra_L \\
& =\la z_1 \, , \, \la x_1 \, , \, (f_{[X, Y]}\circ f_{[Z, W]})(\phi)(c)\cdot x_2 \ra_B \cdot z_2 \ra_L \\
& =\la z_1 \, , \, f_{[Z, W]}(\phi)(\la x_1 \, , \, c\cdot x_2 \ra_D \, )\cdot z_2 \ra_L \\
& =\phi(\la z_1 \, , \, \la x_1 \, , \, c\cdot x_2 \ra_D \cdot z_2 \ra_M ) .
\end{align*}
On the other hand,
\begin{align*}
& \la x_1 \otimes z_1 \, , \, (f_{[X\otimes_B Z\, , \, Y\otimes_D W]})(\phi)(c)\cdot x_2 \otimes z_2 \ra_L \\
& = \phi(\la x_1 \otimes z_1 \, , \, c\cdot x_2 \otimes z_2 \ra_L \, ) \\
& =\phi(\la z_1 \, , \, \la x_1 \, , \, c\cdot x_2 \ra_D \cdot z_2 \ra_M ) .
\end{align*}
By Lemma \ref{lem:con5},
$$
(f_{[X, Y]}\circ f_{[Z, W]})(\phi)=f_{[X\otimes_B Z\, , \, Y\otimes_D W]}(\phi)
$$
for any $\phi\in {}_L \BB_L (M, L)$. Therefore, we obtain the conclusion.
\end{proof}

\section{Strong Morita equivalence}\label{sec:SM} Let $A\subset C$ and $B\subset D$ be inclusions
of $C^*$-algebras with $\overline{AC}=C$ and $\overline{BD}=D$. Let $\psi\in {}_A \BB_A (C, A)$ and
$\phi\in {}_B \BB_B (D, B)$.

\begin{Def}\label{def:SM1}
We say that $\psi$ and $\phi$ are
\sl
strongly Morita equivalent
\rm
if there is an element $(X, Y)\in \Equi(A, C, B, D)$ such that $f_{[X, Y]}(\phi)=\psi$.
Also, we say that $\phi$ and $\psi$ are strongly Morita equivalent with respect to
$(X, Y)$ in $\Equi(A, C, B, D)$.
\end{Def}

\begin{remark}\label{remark:SM1-2} By Lemma \ref{lem:con7}, strong Morita equivalence for bimodule linear maps are
equivlence relation.
\end{remark}

Let $\psi\in{}_A \BB_A (C, A)$ and $\phi\in{}_B \BB_B (D, B)$. We suppose that $\psi$ and $\phi$ are strongly
Morita equivalent with respect to $(X, Y)$ in $\Equi(A, C, B, D)$. Let $L_X$ and $L_Y$ be the linking $C^*$-algebras
for $X$ and $Y$, respectively. Then in the same way as in \cite [Section 3]{Kodaka:Picard2} or Brown, Green and
Rieffel \cite [Theorem 1.1]{BGR:linking}, $L_X$ is a $C^*$-subalgebra of $L_Y$ and by easy computations,
$\overline{L_X L_Y}=L_Y$. Furthermore, there are full projections $p, q\in M(L_X )$ with $p+q=1_{M(L_X )}$
satisfying the following conditions:
\begin{align*}
pL_X p & \cong A , \quad pL_Y p \cong C , \\
qL_X q & \cong B , \quad qL_Y q \cong D
\end{align*}
as $C^*$-algebras. We note that $M(L_X )\subset M(L_Y )$ by Pedersen \cite [Section 3.12.12]{Pedersen:auto}
since $\overline{L_X L_Y}=L_Y$. 
\par
Let $\phi$, $\psi$ be as above. We suppose that $\phi$ and $\psi$ are selfadjoint. Let $\tau$ be the unique bounded
linear map from $Y$ to $X$ satisfying Conditions (1)-(6) in Proposition \ref{prop:con4}. Let $\rho$
be the map from $L_Y$ to $L_X$ defined by
$$
\rho(\begin{bmatrix} c & y \\
\widetilde{z} & d \end{bmatrix})
=\begin{bmatrix} \psi(c) & \tau(y) \\
\widetilde{\tau(z)} & \phi(d) \end{bmatrix}
$$
for any $c\in C$, $d\in D$, $y, z\in Y$. By routine computations $\rho$ is a selfadjoint element in
${}_{L_X} \BB_{L_X} (L_Y , L_X )$, where ${}_{L_X} \BB_{L_X} (L_Y , L_X )$ is the space of all bounded
$L_X -L_X$-bimodule linear maps from $L_Y$ to $L_X$. Furthermore, $\rho|_{pL_Y p}=\psi$ and $\rho|_{qL_Y q}=\phi$,
where we identify $A, C$ and $B, D$ with $pL_X p$, $pL_Y p$ and $qL_X q$, $qL_Y q$ in the usual way, respectively.
Thus we obtain the following lemma:

\begin{lemma}\label{lem:SM2} With the above notation, let $\psi\in{}_A \BB_A (C, A)$ and $\phi\in {}_B \BB_B (D, B)$.
We suppose that $\psi$ and $\phi$ are selfadjoint and strongly Morita equivalent with respect to
$(X, Y)\in\Equi(A, C, B, D)$. Then there is a selfadjoint
element $\rho\in{}_{L_X} \BB_{L_X} (L_Y , L_X )$ such that
$$
\rho|_{pL_Y p}=\psi , \quad \rho|_{qL_Y q}=\phi .
$$
\end{lemma}
Also, we have the converse direction:

\begin{lemma}\label{lem:SM3} Let $A\subset C$ and $B\subset D$ be as above and let $\psi\in{}_A \BB_A (C, A)$
and $\phi\in{}_B \BB_B (D, B)$ be selfadjoint elements. We suppose that there are an inclusion $K\subset L$ of
$C^*$-algebras with $\overline{KL}=L$ and full projections $p, q\in M(K)$ with $p+q=1_{M(K)}$ such that
$$
A\cong pKp , \quad C\cong pLp, \quad B\cong qKq , \quad D\cong qLq ,
$$
as $C^*$-algebras. Also, we suppose that there is a selfadjoint element $\rho$ in ${}_K \BB_K (L, K)$ such that
$$
\rho|_{pLp}=\psi , \quad \rho|_{qLq}=\phi .
$$
Then $\phi$ and $\psi$ are strongly Morita equivalent, where we identify $pKp$, $pLp$ and $qKq$, $qLq$ with
$A, C$ and $B, D$, respectively.
\end{lemma}
\begin{proof} We note that $(Kp, Lp )\in\Equi(K, L, A, C)$, where we identify $A$ and $C$ with $pKp$ and $pLp$,
respectively. By routine computations, we can see that
$$
\la kp \, , \, \rho(l)\cdot k_1 p \ra_A =\psi(\la kp \, , \, l\cdot k_1 p \ra_C )
$$
for any $k, k_1 \in K$, $l\in L$. Thus by Lemma \ref{lem:con5},
$f_{[Kp \, , \, Lp ]}(\psi)=\rho$. Similarly, $f_{[Kq \, , \, Lq]}(\phi)=\rho$.
Since $f_{[Kq \, , \, Lq ]}^{-1}(\rho)=\phi$, 
$$
(f_{[Kq \, ,\, Lq]}^{-1} \circ f_{[Kp \, , \, Lp ]})(\psi)=\phi .
$$
Since $f_{[Kq \,  ,\, Lq ]}^{-1}=f_{[qK \, , \, qL ]}$, by Lemma \ref{lem:con7}
$$
\phi=f_{[qK \, , \, qL][Kp \, , \, Lp]}(\rho)=f_{[qKp \, ,\, qLp]}(\psi) .
$$
Therefore, we obtain the conclusion.
\end{proof}

\begin{prop}\label{prop:SM4} Let $A\subset C$ and $B\subset D$ be inclusions of $C^*$-algebras
with $\overline{AC}=C$ and $\overline{BD}=D$. Let $\psi$ and $\phi$ be selfadjoint elements in
${}_A \BB_A (C, A)$ and ${}_B \BB_B (D, B)$, respectively. Then the following conditions are equivalent:
\newline
$(1)$ $\psi$ and $\phi$ are strongly Morita equivalent,
\newline
$(2)$ There are an inclusion $K\subset L$ of $C^*$-algebras with $\overline{KL}=L$, full projections
$p, q \in M(K)$ with $p+q=1_{M(K)}$ and a selfadjoint element $\rho \in {}_K \BB_K (L, K)$
satisfying that
$$
A\cong pKp, \quad C\cong pLp, \quad B\cong qKq, \quad D\cong qLq, 
$$
as $C^*$-algebras and that
$$
\rho|_{pLp}=\psi, \quad \rho|_{qLq}=\phi ,
$$
where we identify $pKp$, $pLp$ and $qKq$, $qLq$ with $A$, $C$ and $B$, $D$, respectively.
\end{prop}
\begin{proof} This is immediate by Lemmas \ref{lem:SM2} and \ref{lem:SM3}.
\end{proof}

\section{Stable $C^*$-algebras and matrix algebras}\label{sec:St} Let $A\subset C$ be an inclusion of $C^*$-algebras
with $\overline{AC}=C$. Let $A^s =A\otimes\BK$ and $C^s =C\otimes\BK$.
Let $\{e_{ij}\}_{i, j=1}^{\infty}$ be a system of matrix units of $\BK$. Clearly $A^s \subset C^s$ and
$A\subset C$ are strongly Morita equivalent with respect to the $C^s -C$-equivalence bimodule
$C^s (1_{M(A)}\otimes e_{11})$ and its closed subspace $A^s (1_{M(A)}\otimes e_{11})$, where we
identify $A$ and $C$ with $(1\otimes e_{11})A^s (1\otimes e_{11})$ and $(1\otimes e_{11})C^s (1\otimes e_{11})$,
respectively.

\begin{lemma}\label{lem:St1} With the above notation, for any $\phi\in{}_A \BB_A (C, A)$, 
$$
f_{[A^s (1\otimes e_{11}) \, , \, C^s (1\otimes e_{11})]}(\phi)=\phi\otimes\id_{\BK} .
$$
\end{lemma}
\begin{proof} It suffices to show that
$$
\la a(1\otimes e_{11}) \, , \, (\phi\otimes\id_{\BK})(c)\cdot b(1\otimes e_{11}) \ra_A
=\phi(\la a(1\otimes e_{11}) \, , \, c\cdot b(1\otimes e_{11}) \ra_C )
$$
for any $a, b\in A^s$, $c\in C^s$ by Lemma \ref{lem:con5}. Indeed, for any
$a, b\in A^s$, $c\in C^s$,
\begin{align*}
\la a(1\otimes e_{11})\, , \, (\phi\otimes\id_{\BK})(c)\cdot b(1\otimes e_{11}) \ra_A & =
(1\otimes e_{11})a^* (\phi\otimes\id_{\BK})(c)b(1\otimes e_{11}) \\
& =(\phi\otimes\id_{\BK})((1\otimes e_{11})a^* cb(1\otimes e_{11})) .
\end{align*}
On the other hand,
$$
\phi(\la a(1\otimes e_{11}) \, , \, c\cdot b(1\otimes e_{11}) \ra_C )
=\phi((1\otimes e_{11})a^* cb(1\otimes e_{11})) .
$$
Since we identify $C$ with $(1\otimes e_{11})C^s (1\otimes e_{11})$,
$$
\la a(1\otimes e_{11})\, , \, (\phi\otimes\id_{\BK})(c)\cdot b(1\otimes e_{11}) \ra_A
=\phi(\la a(1\otimes e_{11}) \, , \, c\cdot b(1\otimes e_{11}) \ra_C )
$$
for any $a, b\in A^s$, $c\in C^s$. Therefore, we obtain the conclusion.
\end{proof}

Let $\psi\in{}_A \BB_A (C, A)$. Let $\{u_{\lambda}\}_{\lambda\in\Lambda}$ be an
approximate unit of $A^s$ with $||u_{\lambda}||\leq 1$ for any $\lambda\in\Lambda$.
Since $\overline{AC}=C$, $\{u_{\lambda}\}_{\lambda\in\Lambda}$ is an approximate unit
of $C^s$. Let $c$ be any element in $M(C)$. For any $a\in A$, $\{a\psi(cu_{\lambda})\}_{\lambda\in\Lambda}$
and $\{\psi(cu_{\lambda})a\}_{\lambda\in\Lambda}$ are Cauchy nets in $A$.
Hence there is an element $x\in M(A)$ such that
$\{\psi(cu_{\lambda})\}_{\lambda\in\Lambda}$ is strictly convergent to $x\in M(A)$.
Let $\underline{\psi}$ be the map from $M(C)$ to $M(A)$ defined by
$\underline{\psi}(c)=x$ for any $c\in C$. By routine computations $\underline{\psi}$ is a bounded
$M(A)-M(A)$-bimodule linear map from $M(C)$ to $M(A)$ and $\psi=\underline{\psi}|_C$.
We note that $\underline{\psi}$ is independent of the choice of approximate unit of $A^s$.
\par
Let $q$ be a full projection in $M(A)$, that is, $\overline{AqA}=A$. Since $\overline{AC}=C$,
$M(A)\subset M(C)$ by \cite [Section 3.12.12]{Pedersen:auto}. Thus
$$
\overline{CqC}=\overline{CAqAC}=\overline{CAC}=C .
$$
We regard $qC$ and $qA$ as a $qCq-C$-equivalence bimodule and a $qAq- A$-equivalence
bimodule, respectively.
Then $(qA, qC)\in\Equi(qAq, qCq, A, C)$.

\begin{lemma}\label{lem:St2} With the above notation, for any $\psi\in {}_A \BB_A (C, A)$
$$
f_{[qA, qC]}(\psi)=\psi|_{qCq} .
$$
\end{lemma}
\begin{proof} By easy computations, we see that
$$
\la qx \, , \, \psi|_{qCq}(c)\cdot qz \ra_A =\psi(\la qx \, , \, c\cdot qz \ra_C )
$$
for any $x, z \in A$, $c\in C$ since $\underline{\psi}(q)=q$. Thus we obtain the conclusion
by Lemma \ref{lem:con5}.
\end{proof}

Let $A\subset C$ and $B\subset D$ be inclusions of $C^*$-algebras such that
$A$ and $B$ are $\sigma$-unital and $\overline{AC}=C$ and $\overline{BD}=D$.
Let $B^s =B\otimes\BK$ and $D^s =D\otimes\BK$. We suppose that $A\subset C$ and
$B\subset D$ are strongly Morita equivalent with respect to $(X, Y)\in\Equi (A, C, B, D)$.
Let $X^s =X\otimes\BK$ and $Y^s =Y\otimes \BK$, an $A^s -B^s$-equivalence bimodule and
a $C^s -D^s$-equivalence bimodule, respectively. We note that $(X^s \, , \, Y^s )\in \Equi(A^s , C^s , B^s , D^s )$.
Let $L_{X^s}$ and $L_{Y^s}$ be the linking $C^*$-algebras for $X^s$ and $Y^s$, respectively.
Let
$$
p_1 =\begin{bmatrix} 1_{M(A^s )} & 0 \\
0 & 0 \end{bmatrix} , \quad
p_2 =\begin{bmatrix} 0 & 0 \\
0 & 1_{M(B^s )} \end{bmatrix} .
$$
Then $p_1$ and $p_2$ are full projections in $M(L_{X^s})$. By easy computations, we can see that
$\overline{L_{X^s}L_{Y^s}}=L_{Y^s}$. Hence by \cite [Section 3.12.12]{Pedersen:auto}, $M(L_{X^s})\subset M(L_{Y^s})$.
Since $p_1$ and $p_2$ are full projections in $M(L_X )$, by Brown \cite [Lemma 2.5]{Brown:hereditary}, there
is a partial isometry $w\in M(L_{X^s} )$ such that $w^* w=p_1$, $ww^* =p_2$. We note that $w\in M(L_{Y^s} )$.
Let $\Psi$ be the map from $p_2 L_{Y^s}p_2$ to $p_1 L_{Y^s}p_1$ defined by
$$
\Psi(\begin{bmatrix} 0 & 0 \\
0 & d \end{bmatrix})=w^* \begin{bmatrix} 0 & 0 \\
0 & d \end{bmatrix}w
$$
for any $d\in D^s$. In the same way as in the discussions of \cite {BGR:linking}, $\Psi$ is an isomorphism
of $p_2 L_{Y^s}p_2 $ onto $p_1 L_{Y^s}p_1$ and $\Psi|_{p_2 L_{X^s}p_2}$ is an isomorphism
of $p_2 L_{X^s}p_2$ onto $p_1 L_{X^s }p_1$. Also, we note the following:
\begin{align*}
p_1 L_{Y^s}p_1 \cong C^s & , \quad p_1 L_{X^s}p_1 \cong A^s \\
p_2 L_{Y^s}p_2 \cong D^s & , \quad p_2 L_{X^s}p_2 \cong B^s 
\end{align*}
as $C^*$-algebras. We identify $A^s$, $C^s$ and $B^s$, $D^s$ with $p_1 L_{X^s}p_1$,
$p_1 L_{Y^s}p_1$ and $p_2 L_{X^s}p_2$, $p_2 L_{Y^s}p_2$, respectively. Also,
we identify $X^s$, $Y^s$ with $p_1 L_{X^s}p_2$, $p_1 L_{Y^s}p_2$.
\par
Let $A_{\Psi}^s$ be the $A^s -B^s$-equivalence bimodule induced by $\Psi|_{B^s}$, that is,
$A_{\Psi}^s =A^s$ as $\BC$-vector spaces. The left $A^s$-action and the $A^s$-valued inner product on
$A_{\Psi}^s$ are defined in the usual way. The right $B^s$-action and the right $B^s$-valued inner product on
$A_{\Psi}^s$ are defined as follows: For any $x, y\in A_{\Psi}^s $, $b\in B^s$,
$$
x\cdot b= x\Psi(b), \quad \la x, y \ra_{B^s}=\Psi^{-1}(x^* y) .
$$
Similarly, we define the $C^s -D^s$-equivalence bimodule $C_{\Psi}^s $ induced by $\Psi$.
We note that $A_{\Psi}^s$ is a closed subspace of $C_{\Psi}^s$ and $(A_{\Psi}^s \, , \, C_{\Psi}^s )
\in\Equi(A^s , C^s , B^s , D^s )$.

\begin{lemma}\label{lem:St3} With the above notation, $(A_{\Psi}^s , C_{\Psi}^s )$ is
equivalent to $(X^s , Y^s )$ in $\Equi(A^s , C^s , B^s , D^s )$.
\end{lemma}
\begin{proof} We can prove this lemma in the same way as in the proof of \cite [Lemma 3.3]{BGR:linking}.
Indeed, let $\pi$ be the map from $Y^s$ to $C_{\Psi}^s$ defined by
$$
\pi(y)=\begin{bmatrix} 0 & y \\0 & 0 \end{bmatrix}w
$$
for any $y\in Y^s$. By routine computations, $\pi$ is a $C^s- D^s$-equivalence bimodule isomorphism
of $Y^s$ onto $C_{\Psi}^s$ and $\pi|_{X^s}$ is a bijection from $X^s$ onto $A^s$. Hence by
\cite [Lemma 3.2]{Kodaka:Picard2}, we obtain the conclusion.
\end{proof}
\begin{lemma}\label{lem:St4} With the above notation, for any $\phi\in{}_{B^s} \BB_{B^s} (D^s , B^s )$,
$$
f_{[X^s , Y^s ]}(\phi)=\Psi\circ\phi\circ\Psi^{-1} \in {}_{A^s}\BB_{A^s}(C^s , A^s ).
$$
\end{lemma}
\begin{proof} We claim that
$$
\la x \, , \, (\Psi\circ\phi\circ\Psi^{-1})(d)\cdot z \ra_{B^s}=\phi(\la x, d\cdot z \ra_{D^s})
$$
for any $\phi\in{}_{B^s} \BB_{B^s}(D^s \, , \, B^s )$, $x, z \in A_{\Psi}^s$, $d\in D^s$.
Indeed,
\begin{align*}
\la x \, , (\Psi\circ\phi\circ\Psi^{-1})(d)\cdot z \ra_{B^s} & =\Psi^{-1}(x^*(\Psi\circ\phi\circ\Psi^{-1})(d)z ) \\
&=\Psi^{-1}(x^* )(\phi\circ\Psi^{-1})(d)\Psi^{-1}(z) .
\end{align*}
On the other hand,
\begin{align*}
\phi(\la x, d\cdot z \ra_{D^s}) & =\phi(\Psi^{-1}(x^* dz))=\phi(\Psi^{-1}(x^* )\Psi^{-1}(d)\Psi^{-1}(z)) \\
& =\Psi^{-1}(x^* )(\phi\circ\Psi^{-1})(d)\Psi^{-1}(z)
\end{align*}
since $\Psi^{-1}(x^* )$, $\Psi^{-1}(z)\in B^s$. Thus
$$
\la x \, , (\Psi\circ\phi\circ\Psi^{-1})(d)\cdot z \ra_{B^s} =\phi(\la x, d\cdot z \ra_{D^s}) 
$$
for any $\phi\in{}_{B^s} \BB_{B^s}(D^s , B^s )$, $x, z \in A_{\Psi}^s$, $d\in D^s$.
Hence by Lemma \ref{lem:con5}, $f_{[A_{\Psi}^s , C_{\Psi}^s ]}(\phi)=\Psi\circ\phi\circ\Psi^{-1}$
for any $\phi\in {}_{B^s} \BB_{B^s}(D^s ,B^s )$. Therefore,
$f_{[X^s , Y^s ]}(\phi)=\Psi\circ\phi\circ\Psi^{-1} $ by Lemmas \ref{lem:con6} and \ref{lem:St3}.
\end{proof}

Let $\underline{\Psi}$ be the strictly continuous isomorphism of $M(D^s )$ onto $M(C^s )$
extending $\Psi$ to $M(D^s )$, which is defined in Jensen and Thomsen \cite [Corollary 1.1.15]{JT:KK}.
Then $\underline{\Psi}|_{M(B^s )}$ is an isomorphism of $M(B^s )$ onto $M(A^s )$.
Let $q=\underline{\Psi}(1\otimes e_{11})$. Then $q$ is a full projection in $M(A^s )$
with $\overline{C^s q C^s }=C^s$ and $qA^s q \cong A$, $qC^s q \cong C$ as $C^*$-algebras.
We identify with $qA^s q$ and $qC^s q$ with $A$ and $C$, respectively.
Then we obtain the following proposition:

\begin{prop}\label{prop:St5} Let $A\subset C$ and $B\subset D$ be inclusions of
$C^*$-algebras such that $A$ and $B$ are $\sigma$-unital and 
$\overline{AC}=C$ and $\overline{BD}=D$. Let $\Psi$ be the isomorphism of $D^s$
onto $C^s$ defined before Lemma \ref{lem:St3} and let $q=\Psi(1\otimes e_{11})$.
Let $(X, Y)\in\Equi(A, C, B, D)$.
For any $\phi\in{}_B \BB_B (D, B)$,
$$
f_{[X, Y]}(\phi)=(\Psi\circ (\phi\otimes\id_{\BK})\circ\Psi^{-1})|_{qC^s q} ,
$$
where we identify $qA^s q$ and $qC^s q$ with $A$ and $C$, respectively.
\end{prop}
\begin{proof} We note that $(1\otimes e_{11})B^s (1\otimes e_{11})$ and $(1\otimes e_{11})D^s (1\otimes e_{11})$
are identified with $B$ and $D$, respectively. Also, we identify $qA^s q$ and $qC^s q$ with $A$ and $C$,
respectively. Thus we see that
$$
[qA^s \otimes_{A^s}X^s \otimes_{B^s}B^s (1\otimes e_{11}) \, , \,
qC^s \otimes_{C^s}Y^s \otimes_{D^s}D^s (1\otimes e_{11})]=[X, Y]
$$
in $\Equi(A, C, B, D)/\!\sim$. Hence by Lemma \ref{lem:con7},
$$
f_{[X, Y]}(\phi)=(f_{[qA^s \, , \, qC^s ]}\circ f_{[X^s \, , \, Y^s]}\circ
f_{[B^s (1\otimes e_{11}) \, , \, D^s (1\otimes e_{11})]})(\phi) .
$$
Therefore, by Lemmas \ref{lem:St1}, \ref{lem:St2} and \ref{lem:St4},
$$
f_{[X, Y]}(\phi)=(\Psi\circ (\phi\otimes\id_{\BK})\circ\Psi^{-1})|_{qC^s q} .
$$
\end{proof}

Next, we consider the case that $C^*$-algebras are unital $C^*$-algebras. 
Let $A\subset C$ and $B\subset D$ be unital inclusions of unital $C^*$-algebras, which are
strongly Morita equivalent with respect to a $C-D$-equivalence bimodule $Y$ and its
closed subspace $X$. By \cite [Section 2]{KT4:morita}, there are a positive integer $n$ and a full
projection $p\in M_n (B)$ such that
$$
A\cong pM_n (B)p , \quad C\cong pM_n (D)p
$$
as $C^*$-algebras and such that
$$
X\cong pM_n (B)(1\otimes e), \quad Y\cong pM_n (D)(1\otimes e )
$$
as $A-B$-equivalence bimodules and $C-D$-equivalence bimodules, respectively, where
$$
e=\begin{bmatrix} 1 & 0 & \cdots & 0 \\
0 & 0 & \cdots & 0 \\
\vdots & \vdots & \ddots & \vdots \\
0 & 0 & \cdots & 0 \end{bmatrix}\in M_n (\BC)
$$
and we identify $A$, $C$ and $B$, $D$ with $pM_n (B)p$, $pM_n (D)p$ and
$(1\otimes e)M_n (B)(1\otimes e)$, $(1\otimes e)M_n (D)(1\otimes e)$, respectively.
We denote the above isomorphisms by
\begin{align*}
\Psi_A & : A\longrightarrow pM_n (B)p , \\
\Psi_C & : C\longrightarrow pM_n (D)p , \\
\Psi_X & : X\longrightarrow pM_n (B)(1\otimes e) , \\
\Psi_Y & : Y\longrightarrow pM_n (D)(1\otimes e),
\end{align*}
respectively. Then we have the same results as in \cite [Lemma 2.6 and Corollary 2.7]{KT4:morita}.
First, we construct a map from ${}_B \BB_B (D, B)$ to the space of all $pM_n (B)p-pM_n (B)$-bimodule
maps, ${}_{pM_n (B)p}\BB_{pM_n (B)p}(pM_n(D)p \, , \, pM_n (B)p)$. Let $\phi\in {}_B \BB_B (D, B)$.
Let $\psi$ be the map from $M_n (D)$ to $M_n (B)$ defined by
$$
\psi(x)=(\phi\otimes\id_{M_n (\BC)})(x)
$$
for any $x\in M_n (D)$. Since $\psi(p)=p$, by easy computations, $\psi$ can be regarded as
an element in ${}_{pM_n (B)p}\BB_{pM_n (B)p}(pM_n(D)p \, , \, pM_n (B)p)$. We denote
by $F$ the map
$$
\phi\in{}_B \BB_B (D, B)\mapsto \psi\in{}_{pM_n (B)p}\BB_{pM_n (B)p}(pM_n (D)p\, , \, pM_n (B)p)
$$
as above.

\begin{remark}\label{remark:St6} We note that the inclusion of unital $C^*$-algebras $B\subset D$ is
strongly Morita equivalent to $pM_n (B)p\subset pM_n (D)p$ with respect to the $pM_n (D)p-D$-equivalence
bimodule $pM_n (D)(1\otimes e)$ and its closed subspace $pM_n (B)(1\otimes e)$,
where we identify $B$ and $D$ with $(1\otimes e)M_n (B)(1\otimes e)$ and $(1\otimes)M_n (D)(1\otimes e)$,
respectively.
\end{remark}

\begin{lemma}\label{lem:St7} With the above notation, let $\phi\in {}_B \BB_B (D, B)$. Then for any $d\in M_n (D)$,
$x, y\in M_n (B)$,
$$
\la px(1\otimes e) \, , \, F(\phi)(pdp)\cdot pz(1\otimes e) \ra_B 
=\phi(\la px(1\otimes e) \, , \, pdp\cdot pz(1\otimes e) \ra_D \, ) .
$$
\end{lemma}
\begin{proof} This can be proved by routine computations. Indeed, for any $d\in M_n (D)$, $x, y\in M_n (B)$,
$$
\la px(1\otimes e) \, , \, F(\phi)(pdp)\cdot pz(1\otimes e) \ra_B
=(1\otimes e)x^* p(\phi\otimes\id)(d)pz(1\otimes e)
$$
On the other hand,
$$
\phi(\la px(1\otimes e) \, , \, pdp\cdot pz(1\otimes e)\ra_D \ )
=\phi((1\otimes e)x^* pdpz(1\otimes e)) .
$$
Since we identify $D$ with $(1\otimes e)M_n (D)(1\otimes e)$,
$$
(1\otimes e)x^* p(\phi\otimes\id)(d)pz(1\otimes e)=\phi((1\otimes e)x^* pdpz(1\otimes e)) .
$$
Thus we obtain the conclusion.
\end{proof}

\begin{lemma}\label{lem:St8} With the above notation, for any $\phi\in {}_B \BB_B (D, B)$,
$$
f_{[X, Y]}(\phi)=\Psi_A^{-1}\circ F(\phi)\circ\Psi_C .
$$
\end{lemma}
\begin{proof} By Lemma \ref{lem:con5}, it suffices to show that
$$
\la x \, , \, (\Psi_A^{-1}\circ F(\phi)\circ\Psi_C )(c)\cdot z \ra_B =\phi(\la x \, , \, c\cdot z \ra_D \, )
$$
for any $c\in C$, $x, z\in X$. Indeed, for any $c\in C$, $x, z\in X$,
\begin{align*}
\la x \, , \, (\Psi_A^{-1}\circ F(\phi)\circ\Psi_C )(c)\cdot z \ra_B
& =\la \Psi_X (x) \, , \, \Psi_X ((\Psi_A^{-1}\circ F(\phi)\circ\Psi_C )(c)\cdot z) \ra_B \\
& (\text{by the same result as \cite [Lemma 2.6 (3)]{KT4:morita}}) \\
& =\la \Psi_X (x) \, , \, (F(\phi)\circ\Psi_C )(c)\cdot \Psi_X (z) \ra_B \\
& (\text{by the same result as \cite [Lemma 2.6 (2)]{KT4:morita}}) \\
& =\phi(\la \Psi_X (x) \, , \, \Psi_C (c)\cdot \Psi_X (z) \ra_D ) \quad (\text{by Lemma \ref{lem:St7}}) \\
& =\phi(\la \Psi_X (x) \, , \, \Psi_X (c\cdot z)\ra_D \, ) \\
& (\text{by the same result as \cite [Corollary 2.7 (2), (5)]{KT4:morita}}) \\
& =\phi(\la x \, , \, c\cdot z \ra_D \, ) \\
& (\text{by the same result as \cite [Corollary 2.7 (3)]{KT4:morita}}) .
\end{align*}
Therefore, we obtain the conclusion.
\end{proof}

\section{Basic properties}\label{sec:BP}
Let $A\subset C$ and $B\subset D$ be inclusions of $C^*$-algebras
with $\overline{AC}=C$ and $\overline{BD}=D$.
We suppose that they are strongly Morita equivalent with respect to $(X, Y)\in\Equi(A, C, B, D)$.
Let ${}_A \BB_A (C, A)$ and ${}_B \BB_B (D, B)$
be as above and let $f_{[X, Y]}$ be the isometric isomorphism of ${}_B \BB_B (D, B)$ onto
${}_A \BB_A (C, A)$ induced by $(X, Y)\in\Equi(A, C, B, D)$ which is defined in Section \ref{sec:con}.
In this section, we give basic properties about $f_{[X, Y]}$.

\begin{lemma}\label{lem:BP1} With the above notation, we have the following:
\newline
$(1)$ For any selfadjoint linear map $\phi\in {}_B \BB_B (D, B)$, $f_{[X, Y]}(\phi)$ is selfadjoint.
\newline
$(2)$ For any positive linear map $\phi\in{}_B \BB_B (D, B)$, $f_{[X, Y]}(\phi)$ is positive.
\end{lemma}
\begin{proof} (1) Let $\phi$ be any selfadjoint linear map in ${}_B \BB_B (D, B)$ and
let $c\in C$, $x, z\in X$. By lemma \ref{lem:con5},
\begin{align*}
 \la x \, , \, f_{[X, Y]}(\phi)(c^* )\cdot z \ra_B & =\phi(\la x \, , \, c^* \cdot z \ra_D )
=\phi(\la c\cdot x \, , \, z \ra_D ) \\
& =\phi(\la z\, , \, c\cdot x \ra_D)^*
=\la z \, , \, f_{[X, Y]}(\phi)(c)\cdot x \ra_B^* \\
& =\la f_{[X, Y]}(\phi)(c)\cdot x \, , \, z \ra_B = \la x \, , \, f_{[X, Y]}(\phi)(c)^* \cdot z \ra_B .
\end{align*}
Hence $f_{[X, Y]}(\phi)(c^* )=f_{[X, Y]}(\phi)(c )^* $ for any $c\in C$.
\newline
(2) Let $\phi$ be any positive linear map in ${}_B \BB_B (D, B)$ and let $c$ be any positive element in $C$.
Then $\la x , c\cdot x \ra_D \geq 0$ for any $x\in X$ by Raeburn and Williams
\cite [Lemma 2.28]{RW:continuous}. Hence $\phi(\la x , c\cdot x \ra_D )\geq 0$ for any $x\in X$. That is,
$\la x\, , \, f_{[X, Y]}(\phi)(c)\cdot x \ra_B \geq 0$ for any $x\in X$. Thus $f_{[X, Y]}(\phi)(c)\geq 0$ by
\cite [Lemma 2.28]{RW:continuous}. Therefore, we obtain the conclusion.
\end{proof}

\begin{prop}\label{prop:BP2} Let $A\subset C$ an $B\subset D$ be as in Lemma \ref{lem:BP1}.
If $\phi$ is a conditional expectation from $D$ onto $B$, then
$f_{[X, Y]}(\phi)$ is a conditional expectation from $C$ onto $A$.
\end{prop}
\begin{proof} Since $\phi(b)=b$ for any $b\in B$, for any $a\in A$, $x, z\in X$,
$$
\la x\, , \, f_{[X, Y]}(\phi)(a)\cdot z \ra_B =\phi(\la x \, , \, a\cdot z \ra_B )
=\la x \, , \, a\cdot z \ra_B
$$
by Lemma \ref{lem:con5}. Thus $f_{[X, Y]}(\phi)(a)=a$ for any $a\in A$. By Proposition \ref{prop:con4} and
Lemma \ref{lem:BP1}, we obtain the conclusion.
\end{proof}

Since $A\subset C$ and $B\subset D$ are strongly Morita equivalent with respect to
$(X, Y)\in\Equi (A, C, B, D)$, $A^s \subset C^s$ and $B^s \subset D^s$ are strongly Morita equivalent with
respect to $(X^s , Y^s )\in\Equi (A^s \, , C^s \, , B^s \, , D^s )$. Let $\phi$ be any element
in ${}_B \BB_B (D, B)$. Then
$$
\phi\otimes\id_{\BK} \in{}_{B^s} \BB_{B^s} (D^s , B^s ) .
$$

\begin{lemma}\label{lem:BP3} With the above notation, for any $\phi\in {}_B \BB_B (D, B)$
$$
f_{[X^s , Y^s ]}(\phi\otimes\id_{\BK})=f_{[X, Y]}(\phi)\otimes\id_{\BK} .
$$
\end{lemma}
\begin{proof} This can be proved by routine computations. Indeed, for any $c\in C$,
$x, z\in X$, $k_1, k_2 , k_3 \in \BK$,
\begin{align*}
& \la x\otimes k_1 \, , \, f_{[X^s , Y^s]}(\phi\otimes\id)(c\otimes k_2 )\cdot z\otimes k_3  \ra_{B^s} \\
& = (\phi\otimes\id)(\la x\otimes k_1 \ , \, c\otimes k_2 \cdot z\otimes k_3 \ra_{B^s} ) \\
& = (\phi\otimes\id)(\la x\otimes k_1 \ , \, c\cdot z \otimes k_2 k_3 \ra_{D^s} ) \\
& = (\phi\otimes\id)(\la x \ , \, c\cdot z \ra_D  \otimes k_1^* k_2 k_3 ) \\
& =\la x \, , \, f_{[X, Y]}(\phi)(c)\cdot z \ra_B \otimes k_1^* k_2 k_3 \\
& =\la x\otimes k_1 \, , \, f_{[X, Y]}(\phi)(c)\otimes k_2 \cdot z\otimes k_3  \ra_{B^s}
\end{align*}
by Lemma \ref{lem:con5}. Therefore we obtain the conclusion by Lemma \ref{lem:con5}.
\end{proof}

\begin{cor}\label{cor:BP3-2} With the above notation, let $n\in\BN$. Then for any
$\phi\in{}_B \BB_B (D, B)$,
$$
f_{[X\otimes M_n (\BC) \, , \, Y\otimes M_n (\BC)]}(\phi\otimes\id)=f_{[X, Y]}(\phi)\otimes\id_{M_n (\BC)} .
$$
\end{cor}

\begin{prop}\label{prop:BP4} With the above notation, let $\phi\in{}_B \BB_B (D, B)$. If
$\phi$ is $n$-positive, then $f_{[X, Y]}(\phi)$ is $n$-positive for any $n\in\BN$.
\end{prop}
\begin{proof} This is immediate by Lemma \ref{lem:BP1} and Corollary \ref{cor:BP3-2}.
\end{proof}

Again, we consider the case that $C^*$-algebras are unital $C^*$-algebras. 
Let $A\subset C$ and $B\subset D$ be unital inclusions of unital $C^*$-algebras, which are
strongly Morita equivalent with respect to a $C-D$-equivalence bimodule $Y$ and its
closed subspace $X$. As mentioned in Section \ref{sec:St}, by \cite [Section 2]{KT4:morita}, there are
a positive integer $n$ and a full projection $p\in M_n (B)$ such that
$$
A\cong pM_n (B)p, \quad C\cong pM_n (D)p
$$
as $C^*$-algebras, respectively. We denote the above isomorphisms by
$$
\Psi_A : A\to pM_n (B)p , \quad \Psi_C : C\to pM_n (D)p ,
$$
respectively. We note that $\Psi_A =\Psi_C |_A$.

\begin{prop}\label{prop:BP5} Let $\phi$ be a conditional expectation from $D$ onto $B$. We suppose
that there is a positive number $t$ such that
$$
\phi(d)\geq td
$$
for any positive element $d\in D$. Then there is a positive number $s$ such that
$$
f_{[X, Y]}(\phi)(c)\geq sc
$$
for any positive element $c\in C$.
\end{prop}
\begin{proof}
We recall the discussions after Proposition \ref{prop:St5}. That is, by Lemma \ref{lem:St8},
$f_{[X, Y]}(\phi)=\Psi_A^{-1}\circ F(\phi)\circ\Psi_C$ for any $\phi\in {}_B \BB_B(D, B)$,
where $F$ is the isometric isomorphism of ${}_B \BB_B (D, B)$
onto ${}_{pM_n (B)p}\BB_{pM_n (B)p}(pM_n (D)p\, , \, pM_n (B)p)$ defined before Remark \ref{remark:St6}.
Also, $\Psi_A$ and $\Psi_C$ are isomorphisms
of $A$ and $C$ onto $pM_n (B)p$ and $pM_n (D)p$ defined after Proposition \ref{prop:St5}, respectively.
We note that $\Psi_C |_A =\Psi_A$. Since there is a positive number $t$ such that $\phi(d)\geq td$
for any positive element $d\in D$, by Frank and Kirchberg \cite [Theorem 1]{FK:conditional},
there is a positive number $s$ such that $(\phi\otimes\id_{M_n (\BC)})(d)\geq sd$ for any positive element
$d\in M_n (D)$. Thus for any $c\in C$,
\begin{align*}
f_{[X, Y]}(\phi)(c^* c) & =(\Psi_A^{-1}\circ F(\phi)\circ\Psi_C )(c^* c) \\
& =\Psi_A^{-1}((\phi\otimes\id_{M_n (\BC)})(\Psi_C (c)^* \Psi_C (c)) \\
& \geq \Psi_A^{-1}(s\Psi_C (c)^* \Psi_C (c)) \\
& = sc^* c
\end{align*}
since $F(\phi)=\phi\otimes\id_{M_n(\BC)}$ and $\Psi_C |_A =\Psi_A$.
Therefore, we obtain the conclusion.
\end{proof}

Following Watatani \cite [Definition 1.11.1]{Watatani:index}, we give the following definition.

\begin{Def}\label{def:BP6} Let $\phi\in {}_B \BB_B (D, B)$. Then a finite set $\{(u_i , v_i )\}_{i=1}^m \subset D\times D$
is called
\sl
a quasi-basis
\rm
for $\phi$ if it satisfies that
$$
d=\sum_{i=1}^m u_i \phi(v_i d)=\sum_{i=1}^m \phi(du_i )v_i
$$
for any $d\in D$.
\end{Def}

\begin{lemma}\label{lem:BP7} With the above notation, let $\phi\in{}_B \BB_B (D, B)$
with a quasi-basis $\{(u_i , v_i )\}_{i=1}^m$. Then the finite set of $D\times D$
$$
\{(p(u_i \otimes I_n )a_j p \, , \, p b_j (v_i \otimes I_n )p )\}_{i=1,2,\dots, m, \, j=1,2,\dots, K}
$$
is a quasi-basis for $F(\phi)$, where $a_1 , a_2, \dots, a_K, b_1, b_2, \dots, b_K$ are elements in
$M_n (B)$ with
$$
\sum_{j=1}^K a_j p b_j =1_{M_n(B)} .
$$
\end{lemma}
\begin{proof} This lemma can be proved in the same way as in \cite [Section 2]{KT4:morita}.
Indeed, $F(\phi)=(\phi\otimes\id_{M_n (\BC)})|_{pM_n (D)p}$ and since $p$ is
a full projection in $M_n (B)$, there are elements $a_1,  a_2, \dots, a_K, b_1, b_2, \dots, b_K$ in
$M_n (B)$ such that $\sum_{j=1}^K a_j pb_j =1_{M_n (B)}$. Then the finite set
$$
\{(p(u_i \otimes I_n )a_j p \, , \, p b_j (v_i \otimes I_n )p )\}_{i=1,2,\dots, m, \, j=1,2,\dots, K}
$$
of $D\times D$ is a quasi-basis for $F(\phi)$ by easy computations.
\end{proof}

\begin{prop}\label{prop:BP8} Let $\phi\in{}_B \BB_B (D, B)$. If there is a quasi-basis for $\phi$, then
there is a quasi-basis for $f_{[X, Y]}(\phi)$.
\end{prop}
\begin{proof} This is immediate by Lemmas \ref{lem:St8} and \ref{lem:BP7}.
\end{proof}

\begin{cor}\label{cor:BP9} Let $E^B$ is a conditional expectation from $D$ onto $B$,
which is of Watatani index-finite type, then $f_{[X, Y]}(E^B )$ is a conditional expectation
from $C$ onto $A$, which is of Watatani index-finite type.
\end{cor}
\begin{proof} This is immediate by Propositions \ref{prop:BP2} and \ref{prop:BP8}.
\end{proof}

\section{The Picard groups}\label{sec:Pg} Let $A\subset C$ be an inclusion of $C^*$-algebras
with $\overline{AC}=C$. Let ${}_A \BB_A (C, A)$ be as above. Let $\Pic(A, C)$ be the Picard
group of the inclusion $A\subset C$.

\begin{Def}\label{def:Pg1} Let $\phi\in {}_A \BB_A (C, A)$. We define $\Pic(\phi)$ by
$$
\Pic (\phi)=\{[X, Y]\in\Pic(A, C) \, | \, f_{[X, Y]}(\phi)=\phi \} .
$$
We call $\Pic(\phi)$ the
\it
Picard group
\rm
of $\phi$.
\end{Def}
Let $B\subset D$ be an inclusion of $C^*$-algebras with $\overline{BD}=D$.
Let $\phi \in {}_B \BB_B (D, B)$ and $\psi\in {}_A \BB_A (C, A)$.

\begin{lemma}\label{lem:Pg2} With the above notation, if $\phi$ and $\psi$ are
strongly Morita equivalent with respect to $(Z, W)\in\Equi (A, C, B, D)$, then
$\Pic(\phi)\cong\Pic(\psi)$ as groups.
\end{lemma}
\begin{proof} Let $g$ be the map from $\Pic(\phi)$ to $\Pic(A, C)$ defined by
$$
g([X, Y])=[Z\otimes_B X \otimes_B \widetilde{Z}\, , \, W\otimes_D Y\otimes_D \widetilde{W} ]
$$
for any $[X, Y]\in\Pic (\phi)$. Then since $f_{[Z, W]}(\phi)=\psi$, by Lemma \ref{lem:con7}
\begin{align*}
f_{[Z\otimes_B X \otimes_B \widetilde{Z}\, , \, W\otimes_D Y\otimes_D \widetilde{W} ]}(\psi)
& =(f_{[Z, W]}\circ f_{[X, Y]} \circ f_{[\widetilde{Z}, \widetilde{W}]})(\psi) \\
& =(f_{[Z, W]}\circ f_{[X, Y]} \circ f_{[Z, W]}^{-1})(\psi)=\psi .
\end{align*}
Hence $[Z\otimes_B X \otimes_B \widetilde{Z}\, , \, W\otimes_D Y\otimes_D \widetilde{W} ]\in\Pic(\psi)$
and by easy computations, we can see that $g$ is an isomorphism of $\Pic(\phi)$ onto $\Pic(\psi)$.
\end{proof}

Let $\phi\in{}_A \BB_A (C, A)$. Let $\alpha$ be an automorphism of $C$ such that
the restriction of $\alpha$ to $A$, $\alpha|_A$
is an automorphism of $A$. Let $\Aut (A, C)$ be the group of all such automorphisms and let
$$
\Aut (A, C, \phi)=\{\alpha\in\Aut(A, C) \, | \, \alpha\circ\phi\circ\alpha^{-1}=\phi \} .
$$
Then $\Aut (A, C, \phi)$ is a subgroup of $\Aut (A, C)$. Let $\pi$ be the homomorphism of $\Aut(A, C)$
to $\Pic(A, C)$ defined by
$$
\pi(\alpha)=[X_{\alpha}, Y_{\alpha}]
$$
for any $\alpha\in\Aut (A, C)$, where $(X_{\alpha}, Y_{\alpha})$ is an element in $\Equi (A, C)$
induced by $\alpha$, which is defined in \cite [Section 3]{Kodaka:Picard2}, where $\Equi (A, C)=\Equi(A, C, A, C)$.
Let $u$ be a unitary element in $M(A)$. Then $u\in M(C)$ and $\Ad(u)\in\Aut(A, C)$ since
$\overline{AC}=C$. Let $\Int (A, C)$ be the group of all such automorphisms in $\Aut (A, C)$.
We note that $\Int (A, C)=\Int (A)$, the subgroup of $\Aut(A)$ of all generalized inner automorphisms of $A$.
Let $\imath$ be the inclusion map of $\Int (A, C)$ to $\Aut (A, C)$.

\begin{lemma}\label{lem:Pg3} With the above notation, let $\phi\in{}_A \BB_A (C, A)$. Then the following
hold:
\newline
$(1)$ For any $\alpha\in\Aut (A, C)$, $f_{[X_{\alpha}, Y_{\alpha}]}(\phi)=\alpha\circ\phi\circ\alpha^{-1}$.
\newline
$(2)$ The map $\pi|_{\Aut(A, C, \phi)}$ is a homomorphism of $\Aut(A, C, \phi)$
to $\Pic(\phi)$, where $\pi|_{\Aut(A, C, \phi)}$ is the restriction of $\pi$ to $\Aut (A, C, \phi)$.
\newline
$(3)$ $\Int(A, C)\subset \Aut(A, C, \phi)$ and the following sequence
$$
1\longrightarrow \Int(A, C)\overset{\imath}{\longrightarrow}\Aut(A, C, \phi)\overset{\pi}{\longrightarrow}\Pic(\phi)
$$
is exact.
\end{lemma}
\begin{proof} (1) Let $\alpha\in\Aut(A, C)$. Then for any $c\in C$, $x, z\in X_{\alpha}$,
\begin{align*}
\la x \, , \, (\alpha\circ\phi\circ\alpha^{-1})(c)\cdot z \ra_A & =\la x \, , \, (\alpha\circ\phi\circ\alpha^{-1})(c)z \ra_A \\
& =\alpha^{-1}(x^* (\alpha\circ\phi\circ\alpha^{-1})(c)z) \\
& =\alpha^{-1}(x^* )(\phi\circ\alpha^{-1})(c)\alpha^{-1}(z) .
\end{align*}
On the other hand,
\begin{align*}
\phi(\la x \, , \, c\cdot z \ra_C ) &=\phi(\alpha^{-1}(x^* cz))=\phi(\alpha^{-1}(x^* )\alpha^{-1}(c)\alpha^{-1}(z)) \\
& =\alpha^{-1}(x^* )(\phi\circ\alpha^{-1})(c)\alpha^{-1}(z) .
\end{align*}
Thus by Lemma \ref{lem:con5}, $f_{[X_{\alpha}, Y_{\alpha}]}(\phi)=\alpha\circ\phi\circ\alpha^{-1}$.
\newline
(2) Let $\alpha$ be any element in $\Aut(A, C, \phi)$. Then by (1),
$f_{[X_{\alpha}, Y_{\alpha}]}(\phi)=\alpha\circ\phi\circ\alpha^{-1}=\phi$. Hence
$[X_{\alpha}, Y_{\alpha}]\in\Pic(\phi)$.
\newline
(3) Let $\Ad(u)\in\Int (A, C)$. Then $u\in M(A)\subset M(C)$.
For any $c\in C$,
$$
(\Ad(u)\circ\phi\circ\Ad(u^*))(c) =u\phi(u^* cu)u^*=uu^* \phi(c)uu^* =\phi(c)
$$
since $\underline{\phi}(u)=u$. Thus $\Int(A, C)\subset \Aut(A, C, \phi)$.
It is clear by \cite [Lemma 3.4]{Kodaka:Picard2} that the sequence
$$
1\longrightarrow \Int(A, C)\overset{\imath}{\longrightarrow}\Aut(A, C, \phi)\overset{\pi}{\longrightarrow}\Pic(\phi)
$$
is exact.
\end{proof}

\begin{prop}\label{prop:Pg4} Let $A\subset C$ be an inclusion of $C^*$-algebras with $\overline{AC}=C$
and we suppose that $A$ is $\sigma$-unital. Let $\phi\in {}_{A^s} \BB_{A^s}(C^s, A^s )$. Then
the sequence
$$
1\longrightarrow \Int(A^s , C^s )\overset{\imath}{\longrightarrow}\Aut(A^s , C^s , \phi)
\overset{\pi}{\longrightarrow}\Pic(\phi)\longrightarrow 1
$$
is exact.
\end{prop}
\begin{proof} It suffices to show that $\pi$ is surjective by Lemma \ref{lem:Pg3} (3). Let
$[X, Y]$ be any element in $\Pic(\phi)$. Then by \cite [Proposition 3.5]{Kodaka:Picard2},
there is an element $\alpha\in \Aut (A^s , C^s )$ such that
$$
\pi(\alpha)=[X, Y]
$$
in $\Pic(A, C)$. Since $[X, Y]\in\Pic(\phi)$, $f_{[X, Y]}(\phi)=\phi$. Also, by Lemma \ref{lem:con6},
$f_{[X, Y]}=f_{[X_{\alpha}, Y_{\alpha}]}$, where $[X_{\alpha}, Y_{\alpha}]$ is the element in
$\Pic(A, C)$ induced by $\alpha$. Hence
$$
f_{[X_{\alpha}, Y_{\alpha}]}(\phi)=f_{[X, Y]}(\phi)=\phi .
$$
Since $f_{[X_{\alpha}, Y_{\alpha}]}(\phi)=\alpha\circ\phi\circ\alpha^{-1}$ by Lemma \ref{lem:Pg3}(1),
$\phi=\alpha\circ\phi\circ\alpha^{-1}$. Hence $\alpha\in\Aut (A^s , C^s , \phi)$.
\end{proof}

\section{The $C^*$-basic construction}\label{sec:basic} Let $A\subset C$ be a unital inclusion of
unital $C^*$-algebras and let $E^A$ be a conditional expectation of Watatani index-finite type from
$C$ onto $A$. Let $e_A$ be the Jones' projection for $E^A$ and $C_1$ the $C^*$-basic construction
for $E^A$. Let $E^C$ be its dual conditional expectation from $C_1$ onto $C$. Let $e_C$ be the
Jones' projection for $E^C$ and $C_2$ the $C^*$-basic construction for $E^C$ Let $E^{C_1}$ be the dual
conditional expectation of $E^C$ from $C_2$ onto $C_1$. Since $E^A$ and $E^C$ are of Watatani
index-finite type, $C$ and $C_1$ can be regarded as a $C_1 -A$-equivalence bimodule and a
$C_2 -C$-equivalence bimodule
induced by $E^A$ and $E^C$, respectively. We suppose that the Watatani index of $E^A$, $\Ind_W (E^A )\in A$.
Then by \cite [Examples]{KT4:morita}, inclusions $A\subset C$ and $C_1 \subset C_2$ are strongly
Morita equivalent with respect to the $C_2 -C$ equivalence bimodule $C_1$ and its closed subspace $C$,
where we regard $C$ as a closed subspace of $C_1$ by the map
$$
\theta_C (x)=\Ind_W (E^A )^{\frac{1}{2}}xe_A
$$
for any $x\in C$ (See \cite [Examples]{KT4:morita}).

\begin{lemma}\label{lem:basic1} With the above notation, we suppose that
$\Ind_W (E^A)\in A$. Then $E^A$ and $E^{C_1}$ are strongly Morita equivalent
with respect to $(C, C_1 )\in \Equi(C_1 , C_2 , A, C)$.
\end{lemma}
\begin{proof} By \cite [Lemma 4.2]{KT4:morita}, $A\subset C$ and $C_1 \subset C_2$ are
strongly Morita equivalent with respect to $(C, C_1 )\in\Equi (C_1 , C_2 , A, C)$. Since we regard
$C$ as a closed subspace of $C_1$ by the linear map $\theta_C$,
we have only to show that
$$
\la x \, , \, E^{C_1} (c_1 e_A c_2 e_C d_1 e_A d_2 )\cdot z \ra_A
=E^A (\la \theta_C (x) \, , \, c_1 e_A c_2 e_C d_1 e_A d_2 \cdot \theta_C (z) \ra_C )
$$
for any $c_1, c_2 , d_1 , d_2 \in C$, $x, z\in C$. Indeed,
\begin{align*}
\la x \, , \, E^{C_1} (c_1 e_A c_2 e_C d_1 e_A d_2 )\cdot z \ra_A 
& =\la x \, , \, \Ind_W (E^A )^{-1}c_1 e_A c_2 d_1 e_A d_2 \cdot z \ra_A \\
& =\Ind_W (E^A )^{-1} \la x \, , \, c_1 E^A ( c_2 d_1 ) E^A (d_2 z) \ra_A \\
& =\Ind_W (E^A )^{-1}E^A (x^* c_1 )E^A (c_2 d_1 )E^A (d_2 z) .
\end{align*}
for any $c_1, c_2 , d_1 , d_2 \in C$, $x, z\in C$. On the other hand,
\begin{align*}
E^A (\la \theta_C (x) \, , & \, c_1 e_A c_2 e_C d_1 e_A d_2 \cdot \theta_C (z) \ra_C ) \\
& =\Ind_W (E^A )E^A (\la xe_A \, , \, c_1 e_A c_2 E^C (d_1 e_A d_2 ze_A ) \ra_C ) \\
& =E^A (\la xe_A \, , \, c_1 e_A c_2 d_1 E^A (d_2 z) \ra_C ) \\
& =E^A (E^C (e_A x^* c_1 e_A c_2 d_1 ))E^A (d_2 z) \\
& =E^A (x^* c_1 )E^A (E^C (e_A c_2 d_1 ))E^A (d_2 z) \\
& =\Ind_W (E^A )^{-1}E^A (x^* c_1 )E^A (c_2 d_1 )E^A (d_2 z) .
\end{align*}
Hence
$$
\la x \, , \, E^{C_1} (c_1 e_A c_2 e_C d_1 e_A d_2 )\cdot z \ra_A
=E^A (\la \theta_C (x) \, , \, c_1 e_A c_2 e_C d_1 e_A d_2 \cdot \theta_C (z) \ra_C )
$$
for any $c_1, c_2 , d_1 , d_2 \in C$, $x, z\in C$. 
Thus by Lemma \ref{lem:con5}, $f_{[C, C_1 ]}(E^A )=E^{C_1}$.
Therefore, we obtain the conclusion.
\end{proof}

Let $B\subset D$ be another unital inclusion of unital $C^*$-algebras
and let $E^B$ be a conditional expectation of Watatani index-finite type from
$D$ onto $B$. Let $e_B, D_1 , E^D$, $e_D , D_2 , E^{D_1}$ be as above.

\begin{lemma}\label{lem:basic2} With the above notation, we suppose that
$E^A$ and $E^B$ are strongly Morita equivalent with respect to $(X, Y)\in\Equi (A, C, B, D)$.
Then $E^C$ and $E^D$ are strongly Morita equivalent.
\end{lemma}
\begin{proof} Since $E^A$ and $E^B$ are strongly Morita equivalent with respect to
$(X, Y)\in\Equi(A, C, B, D)$, there is the unique linear map $E^X$
from $Y$ to $X$, which is called a conditional expectation from $Y$ onto $X$ satisfying
Conditions (1)-(6) in \cite[Definition 2.4]{KT4:morita}. Let $Y_1$ be the upward basic construction
of $Y$ for $E^X$ defined in \cite[Definition 6.5]{KT4:morita}. Then by \cite[Corollary 6.3 and Lemma 6.4]
{KT4:morita}, $f_{[Y, Y_1]}(E^D )=E^C$, that is, $E^C$ and $E^D$ are strongly Morita equivalent
with respect to $(Y, Y_1 )\in\Equi(C, C_1 , D, D_1 )$.
\end{proof}

\begin{lemma}\label{lem:basic3} With the above notation, we suppose that $\Ind_W(E^A )\in A$. If
$E^C$ and $E^D$ are strongly Morita equivalent with respect to $(Y, Z)\in\Equi(C, C_1 , D, D_1 )$, then
$E^A$ and $E^B$ are strongly Morita equivalent.
\end{lemma}
\begin{proof} By Lemma \ref{lem:basic2}, there is an element $(Z, Z_1 )\in\Equi(C_1 , C_2 , D_1 , D_2 )$
such that $f_{[Z, Z_1 ]}(E^{D_1})=E^{C_1}$. Since $\Ind_W (E^B )\in B$ by \cite [Lemma 6.7]{KT4:morita},
$$
f_{[C, C_1 ]}(E^A )=E^{C_1} , \quad f_{[D, D_1 ]}(E^B )=E^{D_1}
$$
by Lemma \ref{lem:basic1}. Thus
$$
[\widetilde{C}\otimes_{C_1} Z\otimes_{D_1} D \, , \,
\widetilde{C_1}\otimes_{C_2}Z_1 \otimes_{D_2}D_1 ]\in\Equi(A, C, B, D)
$$
and
$$
f_{[\widetilde{C}\otimes_{C_1} Z\otimes_{D_1} D \, , \,
\widetilde{C_1}\otimes_{C_2}Z_1 \otimes_{D_2}D_1 ]}(E^B )  =
(f_{[C, C_1 ]}^{-1}\circ f_{[Z, Z_1 ]}\circ f_{[D, D_1 ]})(E^B )
=E^A
$$
by Lemma \ref{lem:con7}. Therefore, we obtain the conclusion.
\end{proof}

\begin{prop}\label{prop:basic4} Let $A\subset C$ and $B\subset D$ be unital inclusions of unital
$C^*$-algebras. Let $E^A$ and $E^B$ be conditional expectations from $C$ and $D$ onto $A$ and $B$,
which are of Watatani index-finite type, respectively. Let $E^C$ and $E^D$ be the dual conditional
expectations of $E^A$ and $E^B$, respectively. We suppose that $\Ind_W (E^A )\in A$. Then
the following conditions are equivalent:
\newline
$(1)$ $E^A$ and $E^B$ are strongly Morita equivalent,
\newline
$(2)$ $E^C$ and $E^D$ are strongly Morita equivalent.
\end{prop}
\begin{proof} This is immediate by Lemmas \ref{lem:basic1} and \ref{lem:basic3}.
\end{proof}

Let $A\subset C$ and $C_1$, $C_2$ be as above. Let $E^A$, $E^C$ and $E^{C_1}$ be also
as above. We suppose that $\Ind_W (E^A )\in A$.
We consider the Picard groups $\Pic(E^A )$ and $\Pic(E^C )$ of $E^A$ and $E^C$, respectively.
For any $[X, Y]\in\Pic(E^A )$, there is the unique conditional expectation $E^X$ from $Y$ onto $X$
satisfying Conditions (1)-(6) in \cite [Definition 2.4]{KT4:morita} since $f_{[X, Y]}(E^A )=E^A$.
Let $F$ be the map from $\Pic(E^A )$ to $\Pic(E^C )$ defined by
$$
F(([X, Y])=[Y, Y_1 ]
$$
for any $[X, Y]\in\Pic(E^A )$, where $Y_1$ is the upward basic construction for $E^X$
and by Proposition \ref{prop:basic4}, $[Y, Y_1 ]\in\Pic(E^C )$.
Since $E^X$ is the unique conditional expectation from $Y$ onto $X$ satisfying
Conditions (1)-(6) in \cite [Definition 2.4]{KT4:morita}
we can see that the same results
as \cite [Lemmas 4.3-4.5]{Kodaka:Picard2} hold. Hence in the same way as in the proof of
\cite [Lemma 5.1]{Kodaka:Picard2}, we obtain 
that $F$ is a homomorphism of $\Pic(E^A )$ to $\Pic(E^C )$. Let $G$ be the map from $\Pic(E^A )$
to $\Pic(E^{C_1} )$ defined by for any $[X, Y]\in\Pic(E^A )$
$$
G([X, Y])=[C\otimes_A X\otimes_A \widetilde{C} \, , \, C_1 \otimes_C Y \otimes_C \widetilde{C_1}] ,
$$
where $(C, C_1 )$ is regarded as an element in $\Equi(C_1 , C_2 , A, C)$. By the proof of Lemma \ref{lem:Pg2},
$G$ is an isomorphism of $\Pic(E^A )$ onto $\Pic(E^{C_1} )$. Let $F_1$ be the homomorphism of $\Pic(E^C )$
to $\Pic(E^{C_1})$ defined as above. Then in the same way as in the proof of \cite [Lemma 5.2]{Kodaka:Picard2},
$F_1 \circ F=G$ on $\Pic(E^A )$. Furthermore, in the same way as in the proofs of
\cite [Lemmas 5.3 and 5.4]{Kodaka:Picard2}, we obtain that $F \circ G^{-1}\circ F_1 =\id$ on $\Pic(E^C )$.
Therefore, we obtain the same result as \cite [Theorem 5.5]{Kodaka:Picard2}.

\begin{thm}\label{thm:basic5} Let $A\subset C$ be a unital inclusion of unital $C^*$-algebras.
We suppose that there is a conditional expectation $E^A$ of Watatani index-finite type from
$C$ onto $A$ and that $\Ind_W (E^A )\in A$. Then $\Pic(E^A )\cong \Pic(E^C )$, where $E^C$ is the
dual conditional expectation of $E^A$ from $C_1$ onto $C$ and $C_1$ is the $C^*$-basic
construction for $E^A$.
\end{thm}

\section{Modular automorphisms and relative commutants}\label{sec:Mod}
Following Watatani \cite [Section 1.11]{Watatani:index}, we give the definitions of the modular condition and
modular automorphisms.
\par
Let $A\subset C$ be a unital inclusion of unital $C^*$-algebras. Let $\theta$ be an automorphism of $A' \cap C$ and
$\phi$ an element in ${}_A \BB_A (C, A)$.

\begin{Def}\label{def:Mod1} Let $\theta$ and $\phi$ be as above. $\phi$ is said to satisfy the {\sl modular condition}
for $\theta$ if the following condition holds:
$$
\phi(xy)=\phi(y\theta(x))
$$
for any $x\in A' \cap C$, $y\in C$.
\end{Def}
We have the following theorem which was proved by Watatani in \cite {Watatani:index}:

\begin{thm}\label{thm:Mod2} {\rm (cf: \cite [Theorem 1.11.3]{Watatani:index})} Let $\phi\in {}_A \BB_A (C, A)$
and we suppose that there is a quasi-basis for $\phi$. Then there is the unique automorphism $\theta$ of
$A' \cap C$ for which $\phi$ satisfies the modular condition.
\end{thm}

\begin{Def}\label{def:Mod3} The above automorphism $\theta$ of $A' \cap C$ given by $\phi$ in
Theorem \ref{thm:Mod2} is called the {\sl modular automorphism} associated with $\phi$ and denoted by
$\theta^{\phi}$.
\end{Def}

\begin{remark}\label{remark:Mod4} Following the proof of \cite [Theorem 1.11.3]{Watatani:index},
we give how to construct the modular automorphism $\theta^{\phi}$. Let $\{(u_i , v_i )\}_{i=1}^m$ be
a quasi-basis for $\phi$. Put
$$
\theta^{\phi}(c)=\sum_{i=1}^m u_i \phi(cv_i )
$$
for any $c\in A' \cap C$. Then $\theta^{\phi}$ is the unique automorphism of $A' \cap C$ satisfying the
modular condition by the proof of \cite [Theorem 1.11.3]{Watatani:index}.
\end{remark}

Let $A\subset C$ and $B\subset D$ be unital inclusions of unital $C^*$-algebras, which are strongly
Morita equivalent with respect to a $C-D$-equivalence bimodule $Y$ and its closed subspace $X$.
Then by the discussions after Proposition \ref{prop:St5}
$$
A\cong pM_n (B)p \, , \quad C\cong pM_n (D)p ,
$$
where $n$ is some positive integer and $p$ is a full projection in $M_n (B)$.
Also, there is the isometric isomorphism
$$
F: {}_A \BB_A (D, B)\longrightarrow {}_{pM_n (B)p} \BB_{pM_n (B)p} (pM_n (D)p \, ,\, pM_n (B)p) ,
$$
which is defined after Proposition
\ref{prop:St5}. Furthermore, by the proof of \cite [Lemma 10.3]{KT4:morita}, there is the isomorphism $\pi$
of $B' \cap D$ onto $(pM_n (B)p)' \cap pM_n (D)p$ defined by
$$
\pi(d)=(d\otimes I_n )p
$$
for any $d\in B' \cap D$, where we note that $(pM_n (B)p)' \cap pM_n (D)p=(M_n (B)' \cap M_n (D))p$
and that
$$
M_n (B)' \cap M_n (D)=\{d\otimes I_n \, | \, d\in B' \cap D \}.
$$
Thus we can see that
$$
\pi^{-1}((d\otimes I_n )p)=\sum_{j=1}^K a_j (d\otimes I_n )pb_j =d\otimes I_n
$$
for any $d\in B' \cap D$, where $a_1 , a_2 , \dots , a_K , b_1 , b_2 , \dots , b_K$
are elements in $M_n (B)$ with $\sum_{j=1}^K a_j pb_j =1_{M_n (B)}$ and we identify $M_n (B)' \cap M_n (D)$
with $B' \cap D$ by the isomorphism
$$
B' \cap D \to M_n (B)' \cap M_n (D) : d\to d\otimes I_n .
$$

\begin{lemma}\label{lem:Mod5} With the above notation, let $\phi\in {}_B \BB_B (D, B)$
with a quasi-basis for $\phi$. Then $F(\phi)\in {}_{pM_n (B)p} \BB_{pM_n(B)p}(pM_n (D)p \, , \, pM_n (B)p)$
with a quasi-basis for $F(\phi)$ and
$$
\theta^{F(\phi)}=\pi\circ\theta^{\phi}\circ\pi^{-1} .
$$
\end{lemma}
\begin{proof} Let $\{(u_i , v_i )\}_{i=1}^m$ be a quasi-basis for $\phi$. Then by Lemma \ref{lem:BP7},
$$
\{(p(u_i \otimes I_n )a_j p \, , \, pb_j (v_i \otimes I_n )p)\}_{i=1,2\dots, m, \, j=1,2\dots, K}
$$
is a quasi-basis for $F(\phi)$, where $a_1 , a_2 , \dots , a_K , b_1 , b_2 , \dots , b_K$
are elements in $M_n (B)$
with $\sum_{j=1}^K a_j pb_j =1_{M_n (B)}$.
Then by Remark \ref{remark:Mod4} and the definitions of $F(\phi)$, $\phi$,
for any $d\in B' \cap D$,
\begin{align*}
\theta^{F(\phi)}((d\otimes I_n )p) & =\sum_{i, j}p(u_i \otimes I_n )a_j pF(\phi)((d\otimes I_n )pb_j (v_i \otimes I_n )p) \\
& =\sum_{i, j}p(u_i \otimes I_n )a_j p(\phi\otimes\id)((d\otimes I_n )pb_j (v_i \otimes I_n )p) \\
& =\sum_{i, j}p(u_i \otimes I_n )(\phi\otimes \id)(a_j pb_j (dv_i \otimes I_n ))p \\
& =\sum_i p(u_i \otimes I_n )(\phi\otimes\id)(dv_i \otimes I_n )p \\
& =\sum_i p(u_i \phi(dv_i )\otimes I_n )p \\
& =(\theta^{\phi}(d)\otimes I_n )p .
\end{align*}
On the other hand, for any $d\in B' \cap D$,
$$
(\pi\circ\theta^{\phi}\circ\pi^{-1})((d\otimes I_n )p)=(\pi\circ\theta^{\phi})(d)=(\theta^{\phi}(d)\otimes I_n )p .
$$
Hence
$$
\theta^{F(\phi)}(d)=(\pi\circ\theta^{\phi}\circ\pi^{-1})(d)
$$
for any $d\in(pM_n (B)p)' \cap pM_n (D)p$. Therefore, we obtain the conclusion.
\end{proof}

\begin{thm}\label{thm:Mod6} Let $A\subset C$ and $B\subset D$ be unital inclusions of unital $C^*$-algebras
which are strongly Morita equivalent with respect to a $C-D$-equivalence bimodule $Y$ and
its closed subspace $X$. Let $\phi$ be any element in ${}_B \BB_B (D, B)$ with a quasi-basis for
$\phi$. Let $f_{[X, Y]}$ be the isometric isomorphism of ${}_B \BB_B (D, B)$ onto ${}_A \BB_A (C, A)$ defined in
Section \ref{sec:con}. Then $f_{[X, Y]}(\phi)$ is an element in ${}_A \BB_A (C, A)$ with a quasi-basis for
$f_{[X, Y}](\phi)$ and
there is an isomorphism $\rho$ of $B' \cap D$ onto $A' \cap C$ such that
$$
\theta^{f(\phi)}=\rho\circ\theta^{\phi}\circ\rho^{-1} .
$$
\end{thm}
\begin{proof} By Proposition \ref{prop:BP8}, $f(\phi)\in {}_B \BB_B (D, B)$ with a quasi-basis for $f(\phi)$.
Also, by Lemma \ref{lem:Mod5},
$$
\theta^{F(\phi)}=\pi\circ\theta^{\phi}\circ\pi^{-1} ,
$$
where $\pi$ is the isomorphism of $B' \cap D$ onto $(pM_n (B)p)' \cap pM_n (D)p$ defined as above and $n$ is
some positive integer, $p$ is a full projection in $M_n (B)$. Let $\Psi_C$ be the isomorphism of $C$ onto
$pM_n (D)p$ defined after Proposition \ref{prop:St5}. Since $\Psi_C |_A$ is an isomorphism of $A$ onto $pM_n (B)p$,
$\Psi_C^{-1}\circ\theta^{F(\phi)}\circ\Psi_C$ can be regarded as an automorphism of $A' \cap C$.
We claim that $f_{[X, Y]}(\phi)$ satisfies the modular condition for $\Psi_C^{-1}\circ\theta^{F(\phi)}\circ\Psi_C$.
Indeed, by Lemma \ref{lem:St8}, for any $x\in A' \cap C$, $y\in C$,
\begin{align*}
f_{[X, Y]}(\phi)(xy) & =(\Psi_C^{-1}\circ F(\phi)\circ\Psi_C )(xy) \\
& =(\Psi_C^{-1}\circ F(\phi))(\Psi_C (x)\Psi_C (y)) \\
& =(\Psi_C^{-1}\circ F(\phi))(\Psi_C (y)\theta^{F(\phi)}(\Psi_C (x))) \\
& =(\Psi_C^{-1}\circ F(\phi)\circ\Psi_C )(y(\Psi_C^{-1}\circ\theta^{F(\phi)}\circ\Psi_C )(x)) \\
& =f_{[X, Y]}(\phi)(y(\Psi_C^{-1}\circ\theta^{F(\phi)}\circ\Psi_C )(x)) .
\end{align*}
Hence by Theorem \ref{thm:Mod2}, $\theta^{f_{[X, Y]}(\phi)}=\Psi_C^{-1}\circ\theta^{F(\phi)}\circ\Psi_C$.
Thus, we obtain the conclusion.
\end{proof}

Let $A\subset C$ and $B\subset D$ be unital inclusions of unital
$C^*$-algebras and let $E^A$ and $E^B$ be conditional expectations of Watatani index-finite type from
$C$ and $D$ onto $A$ and $B$, respectively. We suppose that there is an element $(X, Y)\in\Equi (A, C, B, D)$
such that $E^A$ is strongly Morita equivalent to $E^B$, that is,
$$
f_{[X, Y]}(E^B )=E^A .
$$
For any element $h\in A' \cap C$, let ${}_h E^A$ be defined by
$$
{}_h E^A (c)=E^A (ch)
$$
for any $c\in C$. We also define ${}_k E^B$ in the same way as above for any $k\in B' \cap D$.

\begin{lemma}\label{lem:Mod7} With the above notation, for any $h\in A' \cap C$, there is the unique
element $k\in B' \cap D$ such that
$$
f_{[X, Y]}({}_k E^B )={}_h E^A .
$$
\end{lemma}
\begin{proof} Since $A\subset C$ and $B\subset D$ are strongly Morita equivalent
with respect to $(X, Y)\in\Equi(A, C, B, D)$, there are a positive integer $n\in\BN$ and a projection
$p\in M_n (A)$ with $M_n (A)pM_n (A)=M_n (A)$ and  $M_n (C)pM_n (C)=M_n (C)$ such that
the inclusion $B\subset D$ is regarded as the inclusion $pM_n (A)p\subset pM_n (C)p$ and such that $X$ and $Y$
are identified with $(1\otimes e)M_n (A)p$ and $(1\otimes e)M_n (C)p$ (See \cite[Section 2]{KT4:morita}),
where $M_n (A)$ and $M_n (C)$ are identified with $A\otimes M_n (\BC)$ and $C\otimes M_n (\BC)$, respectively, $e$
is a minimal projection in $M_n (\BC)$ and we identified $A$ and $C$ with
$(1\otimes e)(A\otimes M_n (\BC))(1\otimes e)$ and $(1\otimes e)(C\otimes M_n (\BC))(1\otimes e)$,
respectively. Then we can see that for any $h\in A' \cap C$, there is the unique element $k\in B' \cap D$
such that
$$
h\cdot x=x\cdot k
$$
for any $x\in X$. Indeed, by the above discussions, we may assume that
$B=pM_n (A)p$, $D=pM_n (C)p$, $X=(1\otimes e)M_n (A)p$.
Let $h$ be any element in $A' \cap C$. Then for any $x\in M_n (A)$
\begin{align*}
h\cdot (1\otimes e)xp& =(1\otimes e)(h\otimes I_n )xp
=(1\otimes e)x(h\otimes I_n )p \\
& =(1\otimes e)xp(h\otimes I_n )p=(1\otimes e)xp\cdot (h\otimes I_n )p .
\end{align*}
By the proof of \cite [Lemma 10.3]{KT4:morita}, $(h\otimes I_n )p\in (pM_n (A)p)' \cap pM_n (C)p$.
Thus, for any $h\in A' \cap C$, there is an element $k\in B' \cap D$ such that
$$
h\cdot x=x\cdot k
$$
for any $x\in X$. Next, we suppose that there is another element $k_1\in B' \cap D$
such that $h\cdot x =x\cdot k_1$ for any $x\in X$. Then $(c\cdot x)\cdot k=(c\cdot x)\cdot k_1$ for
any $c\in C$, $x\in X$. Since $C\cdot X=Y$ by \cite [Lemma 10.1]{KT4:morita}, $k=k_1$.
Hence $k$ is unique. Furthermore, for any $x, z\in X$, $c\in C$,
\begin{align*}
\la x \, , \, {}_h E^A (c)\cdot z \ra_B & =\la x\, , \, E^A (ch)\cdot z \ra_B
=E^B (\la x \, , \, ch \cdot z \ra_D )=E^B (\la x \, , \, c\cdot z\cdot k \ra_D ) \\
& =E^B (\la x \, , \, c\cdot z \ra_D k) ={}_k E^B (\la x \, , \, c\cdot z \ra_D ) .
\end{align*}Therefore, we obtain the conclusion by Lemma \ref{lem:con5}.
\end{proof}

\begin{remark}\label{remark:Mod8} Let $\pi$ be the map from $A' \cap C$ to
$(pM_n (A)p)' \cap pM_n (C)p$ defined by $\pi(h)=(h\otimes I_n )p$ for any $h\in A' \cap C$.
Then $\pi$ is an isomorphism of $A' \cap C$ onto $(pM_n (A)p)' \cap pM_n (C)p$ by
the proof of \cite [Lemma 10.3]{KT4:morita}. We regard $\pi$ as an isomorphism of $A' \cap C$
onto $B' \cap D$. By the above proof, we can see that
$k=\pi(h)$. Thus we obtain that $f_{[X, Y]}({}_{\pi(h)}E^B )={}_h E^A$
for any $h\in A' \cap C$, that is, for any $h\in A' \cap C$, ${}_h E^A$ and ${}_{\pi(h)}E^B$ are
strongly Morita equivalent.
\end{remark}

\begin{prop}\label{prop:Mod9} With the above notation, $\Pic({}_h E^A )\cong \Pic({}_{\pi(h)}E^B )$ for any
$h\in A' \cap C$.
\end{prop}
\begin{proof} This is immediate by Lemma \ref{lem:Pg2}
\end{proof}

\begin{cor}\label{cor:Mod10} Let $A\subset C$ be a unital inclusion of unital $C^*$-algebras.
Let $E^A$ be a conditional expectation of Watatani index-finite type
from $C$ onto $A$. Let $[X, Y]\in \Pic(E^A )$. Then there is an automorphism $\alpha$ of
$A' \cap C$ such that
$$
f_{[X, Y]}({}_{\alpha(h)}E^A )={}_h E^A
$$
for any $h\in A' \cap C$.
\end{cor}
\begin{proof} This is immediate by Lemma \ref{lem:Mod7} and Remark \ref{remark:Mod8}.
\end{proof}

Let $\rho_A$ and $\rho_B$ be the (not $*$-) anti-isomorphism of $A' \cap C$
and $B' \cap D$ onto $C' \cap C_1$ and $D' \cap D_1$, which are defined in
\cite [pp.79]{Watatani:index}, respectively. By the discussions as above or the discussions
in \cite [Section 2]{KT4:morita}, there are a positive integer $n$ and a projection $p$ in
$M_n (A)$ satisfying
\begin{align*}
& M_n (A)pM_n (A) =M_n (A), \quad M_n (C)pM_n (C)=M_n (C), \\
& M_n (C_1 )pM_n (C_1 ) =M_n (C_1 ) , \\
& B \cong pM_n (A), \quad D\cong pM_n (C)p, \quad D_1 \cong pM_n (C_1 )p
\end{align*}
as $C^*$-algebras. Then by the proof of \cite [Lemma 10. 3]{KT4:morita},
\begin{align*}
(pM_n (A)p)' \cap pM_n (C)p & =\{(h\otimes I_n )p \, | \, h\in A' \cap C \} , \\
(pM_n (C)p)' \cap pM_n (C_1 )p & =\{(h_1 \otimes I_n )p \, | \, h_1 \in C' \cap C_1  \} .
\end{align*}
And by easy computations, the anti-isomorphism $\rho$ of $(pM_n (A)p)' \cap pM_n (C)p$ onto
$(pM_n (C)p)' \cap pM_n (C_1 )p$ defined in the same way as in \cite[pp.79]{Watatani:index} is
following:
$$
\rho((h\otimes I_n )p)=(\rho_A (h)\otimes I_n )p
$$
for any $h\in A' \cap C$. This proves that $\pi_1 \circ\rho_A =\rho_B \circ\pi$,
where $\pi$ and $\pi_1$ are the isomorphisms of $A' \cap C$ and $C' \cap C_1$ onto
$(pM_n (A)p)' \cap pM_n (C)p$ and $(pM_n (C)p)' \cap pM_n (C_1 )p$
defined in \cite [Lemma 10.3]{KT4:morita}, respectively and we regard $\pi$ and $\pi_1$ as
isomorphisms of $A' \cap C$ and $C' \cap C_1$ onto $B' \cap D$ and $D' \cap D_1$, respectively.
Then we have the following:

\begin{remark}\label{remark:Mod11} (1) If $f_{[X, Y]}(E^B )=E^A$,
then $f_{[Y, Y_1 ]}({}_{\rho_B (\pi(h))}E^D )={}_{\rho_A (h)}E^C $ for any $h\in A' \cap C$. Indeed,
by Lemma \ref{lem:basic2} $f_{[Y, Y_1 ]}(E^D )=E^C$.
Thus by Remark \ref{remark:Mod8},
for any $c\in C' \cap C_1$, $f_{[Y, Y_1 ]}({}_{\pi_1 (c)}E^D )={}_c E^C$.
Hence for any $h\in A' \cap C$,
$$
f_{[Y, Y_1 ]}({}_{\rho_B (\pi(h))}E^D )=f_{[Y, Y_1 ]}({}_{\pi_1 (\rho_A (h))}E^D )={}_{\rho_A (h)}E^C 
$$
since $\pi_1 \circ \rho_A =\rho_B \circ\pi$.
\newline
(2) We suppose that $\Ind_W (E^A )\in A$ and 
$f_{[Y, Y_1 ]}(E^D )=E^C$. Then we can obtain that
$f_{[X, Y]}({}_{(\rho_B^{-1}(\pi_1 ((c))}E^B )={}_{\rho_A^{-1}(c)}E^A$
for any $c\in C' \cap C_1$. In the same way as above, this is immediate by Lemma \ref{lem:basic2}
and by Remark \ref{remark:Mod8}.
\end{remark}

\section{Examples}\label{sec:Ex} In this section, we shall give some easy examples of the
Picard groups of bimodule maps.

\begin{exam}\label{exam:Ex1} Let $A\subset C$ be a unital inclusion of unital $C^*$-algebras and $E^A$
a conditional expectation of Watatani index-finite type from $C$ onto $A$.
We suppose that $A' \cap C=\BC1$. Then $\Pic(E^A )=\Pic(A, C)$.
\end{exam}
\begin{proof} Since $E^A$ is the unique
conditional expectation by \cite [Proposition 1.4.1]{Watatani:index}, for any
$[X, Y]\in\Pic(A, C)$, $f_{[X, Y]}(E^A )=E^A$. Thus $\Pic(E^A )=\Pic(A, C)$.
\end{proof}

Let $(\alpha, w)$ be a twisted action of a countable discrete group $G$
on a unital $C^*$-algebra $A$ and let $A\rtimes_{\alpha, w, r}G$ be the reduced twisted
crossed product of $A$ by $G$.
Let $E^A$ be the canonical conditional expectation from
$A\rtimes_{\alpha, w, r}G$ onto $A$ defined by
$E^A (x)=x(e)$ for any $x\in K(G, A)$, where $K(G, A)$ is the $*$-algebra of all complex valued functions on $G$ with
a finite support and $e$ is the unit element in $G$.

\begin{exam}\label{exam:Ex2} We suppose that the twisted action $(\alpha, w)$ is free.
Then $E^A$ is the unique conditional
expectation from $A\rtimes_{\alpha, w, r}G$ onto $A$ by \cite [Proposition4.1]{Kodaka:countable}.
Hence $\Pic(E^A )=\Pic(A, A\rtimes_{\alpha, w, r}G)$ by the same reason as Example \ref{exam:Ex1}.
\end{exam}

Let $A$ be a unital $C^*$-algebra such that the sequence
$$
1\longrightarrow \Int(A)\longrightarrow \Aut(A)\longrightarrow \Pic(A)\longrightarrow 1
$$
is exact, where $\Aut(A)$ is the group of all automorphisms of $A$ and
$\Int(A)$ is the subgroup of $\Aut(A)$ of all inner automorphisms of $A$.
We consider the unital inclusion of unital $C^*$-algebras $\BC1\subset A$. Let
$\phi$ be a bounded linear functional on $A$. We regard $\phi$ as a $\BC$-bimodule map from $A$ to $\BC$.
Let $\Aut^{\phi}(A)$ be the subgroup of $\Aut(A)$ defined by
$$
\Aut^{\phi}(A)=\{\alpha\in\Aut(A) \, | \, \phi=\phi\circ\alpha\} .
$$
Also, let $U(A)$ be the group of all unitary elements in $A$ and
let $U^{\phi}(A)$ be the subgroup of $U(A)$ defined by
$$
U^{\phi}(A)=\{u\in U(A) \, | \, \phi\circ \Ad(u)=\phi\} .
$$
By \cite [Lemma 7.2 and Example 7.3]{Kodaka:Picard2},
$$
\Pic(\BC1, A)\cong U(A)/U(A' \cap A)\rtimes_s \Pic(A) ,
$$
that is, $\Pic(\BC1, A)$ is isomorphic to a semidirect product group of $U(A)/U(A' \cap A)$ by $\Pic(A)$ and
generated by
$$
\{[\BC u, A]\in\Pic(\BC1, A) \, | \, u\in U(A) \}
$$
and
$$
\{[\BC 1, X_{\alpha} ]\in\Pic(\BC1, A) \, | \, \alpha\in \Aut(A) \} ,
$$
where $X_{\alpha}$ is the $A-A$-equivalence
bimodule induced by $\alpha\in \Aut(A)$ (See \cite[Example 7.3]{Kodaka:Picard2}).

\begin{exam}\label{exam:Ex3} Let $A$ be a unital $C^*$-algebra such that the sequence
$$
1\longrightarrow \Int(A)\longrightarrow \Aut(A)\longrightarrow \Pic(A)\longrightarrow 1
$$
is exact. Let $\phi$ be a bounded linear functional on $A$. Let $\Pic^{\phi}(A)$ be the subgroup
of $\Pic(A)$ defined by
$$
\Pic^{\phi}(A)=\{[X_{\alpha}] \, | \, \alpha\in \Aut^{\phi}(A) \} .
$$
Then $\Pic(\phi)\cong U(A)/U(A' \cap A)\rtimes_s \Pic^{\phi}(A)$.
\end{exam}
\begin{proof} Let $\alpha\in\Aut(A)$. Then by Lemma \ref{lem:Pg3}(1),
$$
f_{[\BC1, X_{\alpha}]}(\phi)=\alpha\circ\phi\circ\alpha^{-1}=\phi\circ\alpha^{-1} .
$$
Hence $\alpha\in\Pic^{\phi}(A)$ if and only if $f_{[\BC1, X_{\alpha}]}(\phi)=\phi$.
Also, by Lemma \ref{lem:con5}, for any $a\in A$,
$$
\la u \, , \, f_{[\BC u, A]}(\phi)(a)\cdot u \ra _{\BC}=\phi( \la u \, , \, a\cdot u \ra_A )=\phi(u^* au) ,
$$
that is, $f_{[\BC u , A]}(\phi)(a)=\phi(\Ad(u^* )(a))$. Hence by \cite [Example 7.3]{Kodaka:Picard2},
$$
\Pic(\phi)\cong U^{\phi}(A)/U(A' \cap A)\rtimes_s \Pic^{\phi}(A) .
$$
\end{proof}

\begin{remark}\label{remark:Ex4} If $\tau$ is the unique tracial state on $A$, $\Pic^{\tau}(A)=\Pic(A)$.
Hence
$$
\Pic(\tau)\cong \Pic(\BC1, A)\cong U(A)/U(A' \cap A)\rtimes_s \Pic(A) .
$$
\end{remark}

Let $A$ be a unital $C^*$-algebra such that the sequence
$$
1\longrightarrow \Int(A)\longrightarrow \Aut(A)\longrightarrow \Pic(A)\longrightarrow 1
$$
is exact. Let $n$ be any positive integer with $n\geq 2$. We consider the unital inclusion of
unital $C^*$-algebras $a\in A \mapsto a\otimes I_n \in M_n (A)$, where $I_n$ is the unit element
in $M_n (A)$. We regard $A$ as a $C^*$-subalgebra of $M_n (A)$ by the above unital inclusion map.
Let $E^A$ be the conditional expectation from $M_n (A)$ onto $A$ defined by
$$
E^A ([a_{ij}]_{i, j=1}^n )=\frac{1}{n}\sum_{i=1}^n a_{ii}
$$
for any $[a_{ij}]_{i, j=1}^n \in M_n (A)$. Let $\Aut_0 (A, M_n (A))$ be the group of all automorphisms
$\beta$ of $M_n (A)$ with $\beta|_A =\id$ on $A$. By \cite [Example 7.6]{Kodaka:Picard2}, 
$$
\Pic(A, M_n (A))\cong\Aut_0 (A, M_n (A))\rtimes_s \Pic(A)
$$
and the sequence
$$
1\longrightarrow \Aut_0 (A, M_n (A))\overset{\imath}\longrightarrow \Pic(A, M_n (A))\overset{f_A}
\longrightarrow\Pic(A)\longrightarrow 1
$$
is exact, where $\imath$ is the inclusion map of $\Aut_0 (A, M_n (A))$ defined by
$$
\imath(\beta)=[A, Y_{\beta}]
$$
for any $\beta\in\Aut_0 (A, M_n (A))$ and $f_A$ is defined by $f_A ([X, Y])=[X]$ for any $[X, Y]\in\Pic(A, M_n (A))$.
Also, let $\jmath$ be the homomorphism of $\Pic(A)$ to $\Pic(A, M_n (A))$ defined by
$\jmath([X_{\alpha}])=[X_{\alpha}, X_{\alpha\otimes\id}]$ for any $\alpha\in\Aut(A)$.

\begin{exam}\label{exam:Ex5} Let $A$ be a unital $C^*$-algebra such that the sequence
$$
1\longrightarrow \Int(A)\longrightarrow \Aut(A)\longrightarrow \Pic(A)\longrightarrow 1
$$
is exact. Let $n$ be any positive integer with $n\geq 2$. Let $E^A$ be as above. Let $\Aut_0^{E^A}(A, M_n (A))$
be the subgroup of $\Aut_0 (A, M_n (A))$ defined by
$$
\Aut_0^{E^A}(A, M_n (A))=\{\beta\in\Aut_0 (A, M_n (A)) \, | \, E^A =E^A \circ \beta \} .
$$
Then $\Pic(E^A )\cong \Aut_0^{E^A}(A, M_n (A))\rtimes_s \Pic(A)$.
\end{exam}
\begin{proof} Let $\beta\in\Aut_0 (A, M_n (A))$. Then by Lemma \ref{lem:Pg3}(1),
$$
f_{[X_{\beta}, Y_{\beta}]}(E^A )=\beta\circ E^A \circ\beta^{-1}=E^A \circ\beta^{-1} .
$$
Hence $\beta\in\Aut_0^{E^A}(A, M_n (A))$ if and only if $f_{[X_{\beta}, Y_{\beta}]}(E^A )=E^A$.
Also, by Lemma \ref{lem:Pg3}(1) for any $\alpha\in\Aut(A)$,
$$
f_{[X_{\alpha}, X_{\alpha\otimes\id}]}(E^A )=\alpha\circ E^A \circ (\alpha^{-1}\otimes\id)=E^A
$$
since we identify $A$ with $A\otimes I_n$. Thus by \cite [Example 7.6]{Kodaka:Picard2},
$$
\Pic(E^A)\cong \Aut_0^{E^A}(A, M_n (A))\rtimes_s \Pic(A) .
$$
\end{proof}

\end{document}